\documentclass[10pt]{amsart}
\usepackage{amsmath,amssymb,amsthm,amsfonts}
\usepackage{mathrsfs,eufrak,euscript}
\usepackage{enumerate}
\usepackage{graphicx}

\usepackage{tikz}
\theoremstyle{plain}
\newtheorem{theorem}[subsection]{Theorem}
\newtheorem{proposition}[subsection]{Proposition}
\newtheorem{lemma}[subsection]{Lemma}
\newtheorem{eg}[subsection]{Example}
\newtheorem{eg's}[subsection]{Examples}

\newtheorem{definition}[subsection]{Definition}
\newtheorem{remark}[subsection]{Remark}

\newtheorem{note}[subsection]{Note}
\newtheorem{corollary}[subsection]{Corollary}
\title[Multiplication form]{Spectral theorem for quaternionic  normal operators : Multiplication form}
\author[G. Ramesh {\protect \and} P. Santhosh Kumar]{G. Ramesh {\protect
		\and} P. Santhosh Kumar}
\address{G. Ramesh, Department of Mathematics, Kandi, Sangareddy, Medak, Telangana, India 502285.}
\email{rameshg@iith.ac.in}

\address{P. Santhosh Kumar, Department of Mathematics, Kandi, Sangareddy, Medak, Telangana, India 502285.}
\email{ma12p1004@iith.ac.in}

\begin{document}
\maketitle	
%%%%%%%%%%%%%%%%%%%%%%%%%%%%%%%%%%%%%%%%%%%%%%%%%%%%%%%%%%%%%%%%%%%%%%%%%%%%%%%%%%%%%%%%%%%%%%%%%%%%%%%%%%%%%%%%%%%%%%%%
\begin{abstract}
Let $\mathcal{H}$ be a right quaternionic Hilbert space and let $T$ be a quaternionic  normal operator with the  domain $\mathcal{D}(T) \subset \mathcal{H}$.  Then for a fixed unit imaginary quaternion  $m$, there exists a Hilbert basis $\mathcal{N}_{m}$ of $\mathcal{H}$, a measure space $(\Omega, \mu)$,     a unitary operator $U \colon \mathcal{H} \to L^{2}(\Omega; \mathbb{H}; \mu)$ and a $\mu$ - measurable function $\phi \colon \Omega \to \mathbb{C}_m$ (here $\mathbb{C}_{m} = \{\alpha + m \beta; \;\alpha, \beta \in \mathbb{R}\}$) such that
\[
Tx = U^{*}M_{\phi}Ux, \; \mbox{for all}\; x\in \mathcal{D}(T),
\]
where $M_{\phi}$ is the multiplication operator on $L^{2}(\Omega; \mathbb{H}; \mu)$ induced by $\phi$ with $ U(\mathcal{D}(T)) \subseteq \mathcal{D}(M_{\phi})$. In the process, we prove that every complex Hilbert space is a slice Hilbert space.
		
We establish these results  by  reducing it to the complex case then lift it to the quaternionic case.
\end{abstract}
%%%%%%%%%%%%%%%%%%%%%%%%%%%%%%%%%%%%%%%%%%%%%%%%%%%%%%%%%%%%%%%%%%%%%%%%%%%%%%%%%%%%%%%%%%%%%%%%%%%%%%%%%%%%%%%%%%%%%%%%%%
\section{Introduction and Preliminaries}
In 1936, Birkhoff and von Neumann \cite{Birkhoff} introduced the idea of formulating quantum mechanics in quaternion setting. Later several authors continued the study of quaternionic Hilbert spaces in various directions (see \cite{Adler, Agarwal, Dunford, Finkelstein, Finkelstein2,Finkelstein3,viswanath} for details). There was no suitable notion of spectrum of quaternionic linear operators until the concept of {\it spherical spectrum} was proposed, in 2007,  by Colombo, Gentile, Sabadini, and Struppa  \cite{Colombo}. By using the concept of {\it spherical spectrum},  Alpay, Colombo and Kimsey \cite{Alpay} proved the spectral theorem for unbounded quaternionic normal operator. In   \cite{Ghiloni,Ghiloni2}, Ghiloni, Moretti and Perotti defined the continuous slice functional calculus and proved spectral theorem in quaternion setting.

In quantum mechanics most of the operators we encounter are unbounded, for example, position operator, momentum operator and Schr\"odinger operator \cite{Naylor}. The similar situation occur in quaternionic setting also. One of the most important operators in quantum mechanic is the position operator, which is nothing but a multiplication operator defined on a Hilbert space. This is a normal operator.
In fact, it is well known, in the classical theory of operators, that every normal operator is a multiplication operator induced by a suitable function. One can ask whether the same is true or not in quaternionic setting. Though this question is addressed in various forms in the literature (see for example \cite{Ghiloni2, viswanath}), we prove a version of the multiplication form of the spectral theorem, which exactly look like the classical one.

We organize this article in four sections. In the first section we recall basic properties of the ring of quaternions, quaternionic Hilbert spaces and quaternionic operators. In the second section, we prove the following results:
\begin{itemize}
	\item every complex Hilbert space is a slice Hilbert space
%	of a quaternionic Hilbert space for some quaternionic Hilbert space, anti self-adjoint and unitary operator
	\item  a linear operator between two  complex Hilbert spaces can be extended to a unique  right linear operator between quaternionic Hilbert spaces, and
	\item multiplication form of the spectral theorem for bounded quaternionic normal operator.
\end{itemize}	
In the final section, we extended the spectral theorem for unbounded quaternionic normal operators  via the bounded transform.
\subsection{Quaternion ring}
The set of all expressions of the form  $q = q_{0}+q_{1}i+q_{2}j+q_{3}k$,
where $q_{\ell} \in \mathbb{R}$ for $\ell = 0,1,2,3$ is denoted by $\mathbb{H}$. Here $i,j,k$ satisfy the following:
\begin{equation}\label{ijk}
i^{2} = j^{2} = k^{2} = -1 = i\cdot j \cdot k.
\end{equation}
The addition of two expressions in $\mathbb{H}$ is same as in $\mathbb{C}$, and the multiplication is given by Equation (\ref{ijk}). Note that $\mathbb{H}$ is a non commutative division ring called quaternion ring and the expressions in $\mathbb{H}$ are called quaternions.
Let $q = q_{0}+q_{1}i+q_{2}j+q_{3}k$. Then the conjugate of $q$ is denoted by $\overline{q}$, is defined by $\overline{q} = q_{0}-q_{1}i-q_{2}j-q_{3}k$.  The real part of $q,\; re(q) := q_{0}$ and the imaginary part of $q,\;  im(q) = q_{1}i+q_{2}j+q_{3}k$. The modulus of $q$ is defined by
\begin{equation} \label{defofmodulus}
|q|= \sqrt{\overline{q} q} =  \sqrt{ \sum\limits_{\ell =  0}^{3} q_{\ell}^{2}}.
\end{equation}
%The real and imaginary part of $\mathbb{H}$ is denoted by Re$(\mathbb{H})$ and Im$(\mathbb{H})$ respectively and defined by
%\[\begin{array}{ll}
%\text{Re}(\mathbb{H})&:= \Big\{q \in \mathbb{H}: \; q = \overline{q}\Big\}= \Big\{q \in \mathbb{H}: \; q_{1} = q_{2} = q_{3} =0\Big\}\\
%& \\
%\text{Im}(\mathbb{H})&:= \Big\{q \in \mathbb{H}: \; q = -\overline{q}\Big\}= \Big\{q \in \mathbb{H}: \; q_{1} =0\Big\}.
%\end{array} \]
The imaginary unit sphere is defined by $\mathbb{S} := \{q \in \mathbb{H}\colon \overline{q}=-q,\; |q|=1\}$. For $m\in \mathbb{S}$,  $ \mathbb{C}_{m}:= \{\alpha + m \beta : \alpha, \beta \in \mathbb{R}\}$ is a real subalgebra of $\mathbb{H}$, called the slice of $\mathbb{H}$. In fact,  $\mathbb{C}_{m}$ is  isomorphic to  the complex field $\mathbb{C}$ through the mapping $\alpha + m \beta \to \alpha + i \beta$.  Note that for $m \neq \pm n \in \mathbb{S}$,  we have $\mathbb{C}_{m} \cap \mathbb{C}_{n} = \mathbb{R}$. Moreover, $\mathbb{H} = \bigcup\limits_{m\in \mathbb{S}}\mathbb{C}_{m}$.  The upper half plane of $\mathbb{C}_{m}$ is defined by $\mathbb{C}_{m}^{+} = \big\{\alpha + m \beta;\; \alpha \in \mathbb{R},\; \beta \geq 0\big\}$. Let $p,q \in \mathbb{H}$. Then the relation defined by $p \sim q \Leftrightarrow p = {s}^{-1}q s, \; \text{for some} \;  s \in \mathbb{H}\setminus \{0\}$ is an equivalence relation \cite{Ghiloni}. The equivalence class of $p$, denoted by $[p]$, is given by
\begin{equation*}
[p] = \Big\{p^{\prime}\colon \; re(p) = re(p^{\prime}),\; |im(p)|=|im(p^{\prime})|\Big\}.
\end{equation*}

%%%%%%%%%%%%%%%%%%%%%%%%%%%%%%%%%%%%%%%%%%%%%%%%%%%%%%%%%%%%%%%%%%%%%%%%%%%%%%%%%%%%%%%%%%%%%%%%%%%%%%%%%%%%%%%%%%%
\begin{note}
	All the results in  complex Hilbert spaces holds true  in  $\mathbb{C}_{m}$- Hilbert space, for any $m \in \mathbb{S}$.
\end{note}
%%%%%%%%%%%%%%%%%%%%%%%%%%%%%%%%%%%%%%%%%%%%%%%%%%%%%%%%%%%%%%%%%%%%%%%%%%
%\subsection*{Relation on $\mathbb{H}$:} Define a relation on $\mathbb{H}$ by
%\[p \sim q \Leftrightarrow p = {s}^{-1}q s, \; \text{for some} \; 0 \neq s \in \mathbb{H}.\]
%Here $\sim$ is an equivalence relation on $\mathbb{H}$. The equivalence class of $q$ is given by
%\[[q]  := \Big\{s^{-1}qs :\;  0 \neq s \in \mathbb{H}\Big\}.\]
%If $im(q) \neq 0$, then $q = re(q) +  \frac{im(q)}{|im(q)|}\cdot |im(q)| $. It follows that
%\begin{equation*}
%[q]= re(q) + |im(q)| \cdot \mathbb{S}.
%\end{equation*}
%We have the following observations through the representation of the equivalence classes:
%\begin{align}\label{Relation}
%[q] \cap \mathbb{C}_{m}& = \left\{ re(q) \pm m \cdot\frac{im(q)}{|im (q)|} \right\} , \text{for all}\; m \in \mathbb{S} \\
%p\sim q \Leftrightarrow re (p) &= re(q)\;\text{and}\; |im(p)|=|im(q)|.
%\end{align}
%%%%%%%%%%%%%%%%%%%%%%%%%%%%%%%%%%%%%%%%%%%%%%%%%%%%%%%%%%%%%%%%%%%%%
\begin{definition} \cite[Definition 2.3]{Ghiloni}
	A map $\left\langle \cdot | \cdot \right\rangle \colon \mathcal{H} \times \mathcal{H} \to \mathbb{H}$ is said to be an inner product on a right $\mathbb{H}$- module $\mathcal{H}$  if it satisfy the following properties:
		\begin{enumerate}
		\item  $\left\langle x | x \right\rangle \geq 0$, for all $x \in \mathcal{H}$. In particular, $ \left\langle x | x \right\rangle =0 \Leftrightarrow x = 0$.
		\item $\left\langle x | y + z \cdot q\right\rangle = \left\langle x | y\right\rangle  + \left\langle x | z \right\rangle \cdot q $, for all $x,y \in \mathcal{H}$ and $q \in \mathbb{H}$.
		\item $\left\langle x | y \right\rangle = \overline{\left\langle y | x \right\rangle }$, for all $x,y \in \mathcal{H}$.
	\end{enumerate}
	Moreover, if $\mathcal{H}$ is complete with respect to the  norm defined by $\|x\| := \sqrt{\left\langle x | x\right\rangle}$, for all $x \in \mathcal{H}$, then $\mathcal{H}$ is called a right quaternionic Hilbert space.
\end{definition}
%%%%%%%%%%%%%%%%%%%%%%%%%%%%%%%%%%%%%%%%%%%%%%%%%%%%%%%%%%%%%%%%%%%%%%%%%%%%%%%%%%%%%%%%%%%%%%%%%%%%%%%%%%%%%%%%%%%%%%%%%%%%%%%%%%%%%%%%%%%%
\begin{note}
	Throughout this article $ \mathcal{H}$ denotes a right quaternionic Hilbert space and we call it as quaternionic Hilbert space.
\end{note}
%R%%%%%%%%%%%%%%%%%%%%%%%%%%%%%%%%%%%%%%%%%%%%%%%%%%%%%%%%%%%%%%%%%%%%%%%%%%%%%%%%%%%%%%%%%%%%%%%%%%%%%%%%%%%%%%%%%%%%%%%%%%
\begin{eg}
	Let $(\Omega, \mu)$ be a measure space. Then
	\begin{equation*}
	L^{2}(\Omega; \mathbb{H}; \mu) : = \Big\{f \colon \Omega \to \mathbb{H} \; | \int\limits_{\Omega}|f(x)|^{2}d\mu(x) < \infty\Big\}
	\end{equation*}
	is a right quaternionic Hilbert space with the inner product defined by
	\begin{equation*}
	\left\langle f | g\right\rangle = \int\limits_{\Omega} \overline{f(x)}\cdot g(x) \ d\mu(x).
	\end{equation*}
	Let $m \in \mathbb{S}$. Then
	\begin{equation*}
	L^{2}(\Omega; \mathbb{C}_{m}; \mu) = \Big\{f \colon \Omega \to \mathbb{C}_{m} \; | \int\limits_{\Omega} |f(x)|^{2} d\mu(x) < \infty\Big\}
	\end{equation*}
	is a $\mathbb{C}_{m}$- Hilbert space with the inner product defined by
	\begin{equation*}
	\left\langle f | g\right\rangle = \int\limits_{\Omega} \overline{f(x)}\cdot g(x) \ d\mu(x).
	\end{equation*}
\end{eg}
%%%%%%%%%%%%%%%%%%%%
\begin{definition}
	Let $(\Omega, \mu)$ be a measure space and fix $m\in \mathbb{S}$. If $\phi \colon \Omega \to \mathbb{C}_{m}$ is measurable, then
	\begin{enumerate}
		\item essential supremum of $|\phi|$ is defined by
		\begin{equation*}
		ess\ sup(|\phi|) = \Big\{\alpha\in \mathbb{R} : \mu \big(\{x: |\phi(x)| > \alpha\}\big) = 0 \Big\}.
		\end{equation*}
		\item essential range of $\phi$ is defined by
		\begin{equation*}
		ess\ ran(\phi) := \Big\{\lambda \in \mathbb{C}_{m} : \ \mu \big(\{x : |\phi(x) - \lambda | =0 \}\big)> \epsilon, \forall \ \epsilon >0 \Big\}.
		\end{equation*}
		\item $\phi$ is said to be essentially bounded if $ess \ sup(|\phi|)$ is finite.
	\end{enumerate}
\end{definition}
%%%%%%%%%%%%%%%%%%%%%%%%%%%%%%%%%%%%%%%%%%%%%%%%%%%%%%%%%%%%%%%%%%%%%%%%%%%%%%%%%%%%%%%%%%%%%%%%%%%%%%%%%%%%%%%%%%%%%%%%%%%%%%%%%%%%%%%%%%%%%%%
Now we define Hilbert basis of a right quaternionic Hilbert space (see \cite[Proposition 2.5]{Ghiloni} for details).
\begin{definition}
A subset $\mathcal{N}$ of  $\mathcal{H}$ is said to be a Hilbert basis of $\mathcal{H}$ if, for every $z,z^{\prime} \in \mathcal{N}$, we have $\left\langle z | z^{\prime}\right\rangle = \delta_{z,z^{\prime}}$ and $\left\langle x | y \right\rangle = \sum\limits_{z \in \mathcal{N}} \left\langle x | z\right\rangle \left\langle z | y \right\rangle$ for all $x,y \in \mathcal{H}$.
\end{definition}
Note that every quaternionic Hilbert space $\mathcal{H}$ admits a Hilbert basis $\mathcal{N}$ (see \cite[Proposition 2.6]{Ghiloni}). Moreover, every $x \in \mathcal{H}$ can be uniquely decomposed as follows:
\begin{equation*}
x = \sum\limits_{z \in \mathcal{N}} z \left\langle z | x \right\rangle.
\end{equation*}
\begin{definition}
	A map $T \colon \mathcal{D}(T) \subset \mathcal{H}_{1} \to \mathcal{H}_{2}$  with the domain $\mathcal{D}(T),$ a right linear subspace of $\mathcal{H}_{1}$ is said to be right $\mathbb{H}$- linear or quaternionic operator, if $T(x\cdot q + y) = T(x) \cdot q + T(y)$, for all $x,y \in \mathcal{D}(T)$, $q\in \mathbb{H}$. In particular, $T$ is called
	\begin{enumerate}
		\item densely defined operator, if $\mathcal{D}(T)$ is a dense subspace of $\mathcal{H}_{1}$
		\item closed operator, if the graph of $T$, defined by $\mathcal{G}(T) = \Big\{(x, Tx) | x \in \mathcal{D}(T) \Big\}$, is a closed subspace of $\mathcal{H}_{1}\times \mathcal{H}_{2}$
		\item bounded or continuous operator, if $\|Tx\|_{2} \leq K \|x\|_{1}, \; \text{for all} \; x \in \mathcal{D}(T)$, for some $K>0$.  In this case, the norm of $T$, defined by
		\begin{equation*}
		\|T\| = \sup \Big\{\|Tx\|_{2}:\; x \in \mathcal{D}(T), \|x\|_{1}=1\Big\},
		\end{equation*}
		is finite.
	\end{enumerate}
\end{definition}
We denote the set of all densely defined closed operators and bounded  operators between $\mathcal{H}_{1}$ and $\mathcal{H}_{2}$ by $\mathcal{C}(\mathcal{H}_{1}, \mathcal{H}_{2})$ and $ \mathcal{B}(\mathcal{H}_{1}, \mathcal{H}_{2})$ respectively. In particular, $\mathcal{C}(\mathcal{H},\mathcal{H}) = \mathcal{C}(\mathcal{H})$ and $\mathcal{B}(\mathcal{H},\mathcal{H}) = \mathcal{B}(\mathcal{H})$. Let $S, T \in \mathcal{C}(\mathcal{H})$. Then  $S$ is  a restriction of $T$ (or) $T$ is an extension of $S$, denoted by $S \subset T$, if $\mathcal{D}(S) \subseteq \mathcal{D}(T)$ and $ Sx = Tx $, for all $x \in \mathcal{D}(S)$.
%%%%%%%%%%%%%%%%%%%%%%%%%%%%%%%%%%%%%%%%%%%%%%%%%%%%%%%%%%%%%%%%%%%%%%%%%%%%%%%%%%%%%%%%%%%%%%%%%%%%%%%%%%%%%%%%%%%%%%%%%%%%%%%%%%%%%%%%%%%%%%%
%\begin{definition}\cite[Definition 2.9]{Ghiloni}
%	Let $\big(\mathcal{H}_{1}, \left\langle \cdot | \cdot \right\rangle_{1}\big)$ and $\big(\mathcal{H}_{2}, \left\langle \cdot | \cdot \right\rangle_{2}\big)$ be right quaternionic Hilbert spaces with induces norms $\|\cdot\|_{1}$ and $\|\cdot \|_{2}$ respectively.  A map $T \colon \mathcal{H}_{1} \to \mathcal{H}_{2} $ is said to be a right $\mathbb{H}- $ linear or quaternionic operator if
%	\begin{equation*}
%	T(x\cdot q + y) = T(x) \cdot q + T(y),\; \text{for every}\; x,y \in \mathcal{H}_{1}, q \in \mathbb{H}.
%	\end{equation*}
%	
%	
%	Moreover, if there exist $K>0$ such that
%$\|Tx\|_{2} \leq K \|x\|_{1}, \; \text{for all} \; x \in \mathcal{H}_{1}$, then $T$ is called bounded or continuous. In this case, the norm of $T$, defined by
%	\begin{equation*}
%	\|T\| = \sup \Big\{\|Tx\|_{2}:\; x \in \mathcal{H}_{1}, \|x\|_{1}=1\Big\},
%	\end{equation*}
%	is finite.
%\end{definition}
%%%%%%%%%%%%%%%%%%%%%%%%%%%%%%%%%%%%%%%%%%%%%%%%%%%%%%%%%%%%%%%%%%%%%%%%%%%%%%%%%%%%%%%%%%%%%%%%%%%%%%%%%%%%%%%%%%%%%%%%%%%%%%%%%%%%%%%%%%%%
%If $T \in \mathcal{B}(\mathcal{H}_{1}, \mathcal{H}_{2}),$ the null space and the range space are denoted by $N(T)$ and $R(T)$ respectively.
%%%%%%%%%%%%%%%%%%%%%%%%%%%%%%%%%%%%%%%%%%%%%%%%%%%%%%%%%%%%%%%%%%%%%%%%%%%%%%%%%%%%%%%%%%%%%%%%%%%%%%%%%%%%%%%%%%%%%%%%%%%%%%%%%%%%%%%%%%%
\begin{definition}\cite[Definition 2.12]{Ghiloni}
	Let $ T \in \mathcal{C}(\mathcal{H})$ with the domain $\mathcal{D}(T) \subseteq \mathcal{H}$. Then there exists unique operator $T^{*} \in \mathcal{C}(\mathcal{H})$ with the domain given by \begin{equation*}
	\mathcal{D}(T^*) = \big\{y \in \mathcal{H}:\; x \mapsto \left\langle y|Tx \right\rangle \;\text{is continuous on}\; \mathcal{D}(T)\big\}
	\end{equation*} such that $ \left\langle x | Ty\right\rangle = \left\langle T^{*}x | y\right\rangle $, for all $ x \in \mathcal{D}(T),y \in \mathcal{D}(T^*). $ This operator $ T^{*} $ is called the adjoint of $T.$ Furthermore, $T$ is said to be
	\begin{enumerate}
		\item self-adjoint  $(T^{*}=T)$, if $\mathcal{D}(T)= \mathcal{D}(T^{*})$ and $ T^{*}x = Tx$, for all $x \in \mathcal{D}(T)$
		\item anti self-adjoint $(T^{*}=-T)$, if $\mathcal{D}(T)= \mathcal{D}(T^{*})$ and $ T^{*}x = -Tx$, for all $x \in \mathcal{D}(T)$
		\item positive $(T\geq 0)$, if $T^{*} = T$ and $\left\langle x|Tx \right\rangle \geq 0$, for all $ x \in \mathcal{D}(T)$
		\item normal $( T^{*}T=TT^{*})$, if $ \mathcal{D}(T^{*}T)=\mathcal{D}(TT^{*})$ and $T^{*}Tx=TT^{*}x$ for all $ x \in \mathcal{D}(T^{*}T)$.
	\end{enumerate}
\end{definition}
%%%%%%%%%%%%%%%%%%%%%%%%%%%%%%%%%%%%%%%%%%%%%%%%%%%%%%%%%%%%%%%%%%%%%%%%%%%%%%%%%%%%%%%%%%%%%%%%%%%%%%%%%%%%%%%%%%%%%%%%%%%%%%%%%%%%%%%%%%%%%%%
\begin{eg} \cite[Example 1.1]{viswanath}
	Let $(\Omega, \mu)$ be a $\sigma$- additive measure space, $m \in \mathbb{S}$ and  $\phi \colon \Omega \to \mathbb{C}_{m}$ be  measurable. Define $M_{\phi}\colon \mathcal{D}(M_{\phi})\subseteq L^{2}(\Omega; \mathbb{H}; \mu) \to L^{2}(\Omega; \mathbb{H}; \mu)$ by $M_{\phi}(g)(x) = \phi(x) \cdot g(x),$ for all $g \in \mathcal{D}(M_{\phi})$, where
	\begin{equation*}
	\mathcal{D}(M_{\phi}) = \big\{g \in L^{2}(\Omega; \mathbb{H}; \mu):\; \phi \cdot g \in L^{2}(\Omega; \mathbb{H}; \mu) \big\}.
	\end{equation*}
	Then $M_{\phi}$ is a quaternionic operator. Moreover, $M_{\phi} \in \mathcal{B}(L^{2}(\Omega; \mathbb{H}; \mu))$if and only if $\phi$ is essentially bounded. In this case, $\|M_{\phi}\| = ess \ sup(|\phi|)$.
	\end{eg}
%%%%%%%%%%%%%%%%%%%%%%%%%%%%%%%%%%%%%%%%%%%%%%%%%%%%%%%%%%%%%%%%%%%%%%%%%%%%%%%%%%%%%%%%%%%%%%%%%%%%%%%%%%%%%%%%%%%%%%%%%%%%%%%%%%%%%%%%%%%
%In the forthcoming sections we show that every normal operator on a quaternionic Hilbert space is unitarily equivalent to some  multiplication operator.
%%%%%%%%%%%%%%%%%%%%%%%%%%%%%%%%%%%%%%%%%%%%%%%%%%%%%%%%%%%%%%%%%%%%%%%%%%%%%%%%%%%%%%%%%%%%%%%%%%%%%%%%%%%%%%%%%%%%%%%%%%%%%%%%%%%%%%%%%%
%\begin{definition}
%	Let $T \in \mathcal{B}(\mathcal{H}).$ A closed subspace $M$ of $\mathcal{H}$ is said to be invariant under $T$ if $T(M):= \big\{ Tx \colon x \in M\big\} \subseteq M.$ Moreover, if $M^{\bot}$ is also invariant under $T$ then we say that $M$ is a reducing subspace for $T$.
%\end{definition}
%%%%%%%%%%%%%%%%%%%%%%%%%%%%%%%%%%%%%%%%%%%%%%%%%%%%%%%%%%%%%%%%%%%%%%%%%%%%%%%%%%%%%%

Now we recall the definition of the spherical spectrum of quaternionic operators (see \cite[Definition 4.1]{Ghiloni}).

%%%%%%%%%%%%%%%%%%%%%%%%%%%%%%%%%%%%%%%%%%%%%%%%%%%%
\subsection{Spherical spectrum:}
%%%%%%%%%%%%%%%%%%%%%%%%%%%%%%%%%%%%%%%%%%%%%%%%%%%%%%%%%%%%%%%%%%%%%%%%%%%%%%%%%%%%%%%%%%%%%%%%%%%%%%%%%%%%%%%%%%%%%%%%%%%%%%%%%%%%%%%%%%%%
Let  $ T \colon \mathcal{D}(T) \subseteq \mathcal{H} \to \mathcal{H} $  be a right linear operator with domain $\mathcal D(T)$, a right linear subspace of $\mathcal{H}$ and $ q \in \mathbb{H}.$ Define  $\Delta_{q}(T) \colon \mathcal{D}(T^{2}) \to \mathcal{H}$ by
\begin{equation*}
\Delta_{q}(T):= T^{2}-T(q+\overline{q})+I\cdot|q|^{2}.
\end{equation*}
The spherical resolvent of $T$, denoted by $\rho_{S}(T)$, is defined as the set of all $q \in \mathbb{H}$ satisfying the following three properties:
\begin{enumerate}
	\item $N(\Delta_{q}(T)) = \{0\}$.
	\item $R(\Delta_{q}(T))$ is dense in $\mathcal{H}$.
	\item $\Delta_{q}(T)^{-1}\colon R(\Delta_{q}(T)) \to \mathcal{D}(T^{2})$ is bounded.
\end{enumerate}
Then the spherical spectrum of $T$ is defined by $\sigma_{S}(T): = \mathbb{H} \setminus \rho_{S}(T)$.
\begin{remark} \label{J}
	Let $m\in \mathbb{S}$, $J\in\mathcal{B}(\mathcal{H})$ be anti self-adjoint and unitary that is $J^{*}=-J \;\&\; J^{2}=-I$. Then
	\begin{equation*}
	\mathcal{H}^{Jm}_{\pm}= \Big\{x \in \mathcal{H}: J(x)= \pm x\cdot m\Big\}
	\end{equation*}
	is a $\mathbb{C}_{m}$- Hilbert space, called slice Hilbert space (see \cite[Lemma 3.10]{Ghiloni}\label{directsumofH} for details). Note that considering $\mathcal{H}$ as a $\mathbb{C}_{m}$- Hilbert space, $\mathcal{H}^{Jm}_{\pm}$ are non-zero subspaces of $\mathcal{H}$. Moreover, $\mathcal{H}$ admits the following decomposition:
	\begin{equation*}
	\mathcal{H} = \mathcal{H}^{Jm}_{+} \oplus \mathcal{H}^{Jm}_{-}.
	\end{equation*}
	Furthermore, by \cite[Proposition 3.8(f)]{Ghiloni}, if $\mathcal{N}$ is Hilbert basis of $\mathcal{H}^{Jm}_{+}$, then $\mathcal{N}$ is also a Hilbert basis of $\mathcal{H}$ and $ J(x) = \sum\limits_{z\in \mathcal{N}} z \cdot m \left\langle  z | x \right\rangle$. Since $x\cdot n \in \mathcal{H}^{Jm}_{-}$ ($n \in \mathbb{S}$ such that $mn=-nm$), for all $x \in \mathcal{H}^{Jm}_{+}$,  we have $\left\langle x_{+} | x_{-}\right\rangle + \left\langle x_{-}| x_{+}\right\rangle = 0$, for  $x_{\pm} \in \mathcal{H}^{Jm}_{\pm}$.
\end{remark}
It is observed from Remark \ref{J} that every slice Hilbert space is a $\mathbb{C}_{m}$- Hilbert space, for some $m \in \mathbb{S}$. We prove the converse, that is  every $\mathbb{C}_{m}$- Hilbert space is a slice Hilbert space, in the following section.
%%%%%%%%%%%%%%%%%%%%%%%%%%%%%%%%%%%%%%%%%%%%%%%%%%%%%%%%%%%%
%For every $m \in \mathbb{S}$, $\mathcal{H}^{Jm}_{+}$ is a $\mathbb{C}_{m}$- Hilbert space. Our intension is to prove the converse statement: If $K$ is  $\mathbb{C}_{m}$ - Hilbert space for some $m \in \mathbb{S}$, then there exist a right quaternionic Hilbert space $\mathcal{H}$ and an anti self-adjoint, unitary $J \in \mathcal{B}(\mathcal{H})$ such that $K = \mathcal{H}^{Jm}_{+}$ (see Proposition \ref{slice} of this article).
%%%%%%%%%%%%%%%%%%%%%%%%%%%%%%%%%%%%%%%%%%%%%%%%%%%%%%%%%%%%%%%%%%%%%%%%%
\section{Extension of $\mathbb C_m$-Hilbert space to a quaternionic Hilbert space}
%%%%%%%%%%%%%%%%%%%%%%%%%%%%%%%%%%%%%%%%%%%%%%%%%%%%%%%%%%%%%%%%%%%%%%%%%
It is proved in the literature that   a $\mathbb{C}_{m}$- linear operator $(m \in \mathbb{S})$ on a slice Hilbert space $\mathcal{H}^{Jm}_{+}$ can be extended uniquely to a right linear operator on $\mathcal{H}$  (see \cite{Ghiloni} for details) and the converse is true with some condition. We recall the result here.
\begin{proposition} \cite[Proposition 3.11]{Ghiloni}\label{extension}	
	If $T \colon \mathcal{D}(T) \subset \mathcal{H}^{Jm}_{+} \to \mathcal{H}^{Jm}_{+} $ is a $\mathbb{C}_{m}-$ linear operator, then there exists a unique right $\mathbb{H}-$ linear operator $
	\widetilde{T}\colon \mathcal{D}(\widetilde{T})\subset \mathcal{H} \to \mathcal{H} $
	such that $ \mathcal{D}(\widetilde{T}) \bigcap \mathcal{H}^{Jm}_{+} = \mathcal{D}(T), \ J(\mathcal{D}(\widetilde{T})) \subset \mathcal{D}(\widetilde{T})$ and $ \widetilde{T}(x) = T(x), $ for every $x \in \mathcal{H}^{Jm}_{+}.$ The following facts holds:
	\begin{enumerate}
		\item If $T \in \mathcal{B}(\mathcal{H}^{Jm}_{+})$, then $\widetilde{T}\in \mathcal{B}(\mathcal{H})$ and $\|\widetilde{T}\| = \|T\|$
		\item $J\widetilde{T} = \widetilde{T} J$.
	\end{enumerate}
	On the other hand, let $V \colon \mathcal{D}(V) \to \mathcal{H}$ be a right linear  operator. Then $ V = \widetilde{U} $, for a unique bounded $\mathbb{C}_{m}-$ linear operator $ U \colon \mathcal{D}(V) \bigcap \mathcal{H}^{Jm}_{+} \to \mathcal{H}^{Jm}_{+} $ if and only if $J(\mathcal{D}(V)) \subset \mathcal{D}(V)$ and $JV = VJ$.
	
	Furthermore,
	\begin{enumerate}
		\item If $\overline{\mathcal{D}(T)} = \mathcal{H}^{Jm}_{+}$, then $\overline{\mathcal{D}(\widetilde{T})} = \mathcal{H}$ and  $\big(\widetilde{T}\big)^{*} = \widetilde{T^{*}}$
		\item \label{extnmulti}If $S \colon \mathcal{D}(S) \subset \mathcal{H}^{Jm}_{+} \to \mathcal{H}^{Jm}_{+}$ is $\mathbb{C}_{m}$- linear, then $\widetilde{ST} = \widetilde{S} \widetilde{T}$
		\item \label{extninverse}If $S$ is the inverse of $T$, then $\widetilde{S}$ is the inverse of $\widetilde{T}.$
	\end{enumerate}
\end{proposition}
%%%%%%%%%%%%%%%%%%%%%%%%%%%%%%%%%%%%%%%%%%%%%%%%%%%%%%%%%%%%%%%%%
\begin{remark}\label{extensionrmk}
	In particular, if $T \in \mathcal{B}(\mathcal{H})$ is  normal, then there exist an anti self-adjoint and unitary $J \in \mathcal{B}(\mathcal{H})$ such that $TJ = JT$  (see \cite[Theorem 5.9]{Ghiloni} for details).  Thus Proposition \ref{extension} holds true for quaternionic normal operators. 	
\end{remark}
%%%%%%%%%%%%%%%%%%%%%%%%%%%%%%%%%%%%%%%%%%%%%%%%%%%%%%%%%%%%%%%%%%%%%%%%%%%%%%%%%%%%%%%%%%%%%%%%%%%%%%%%%%%%%%%%%%
In case of unbounded operators, the existence of commuting $J$ is given by \cite[Theorem 7.4]{Ghiloni2}. We give a proof for the same via bounded transform, which is different from the existing proofs in the literature.
\begin{theorem}\cite[Theorem 6.1]{Alpay}\label{Ztransform}
	Let $T \in \mathcal{C}(\mathcal{H})$ and define $\mathcal{Z}_{T}:= T(I+T^{*}T)^{-\frac{1}{2}}$. Then $\mathcal{Z}_{T}$ has the following properties:
	\begin{enumerate}
		\item $\mathcal{Z}_{T} \in \mathcal{B}(\mathcal{H}),\; \|\mathcal{Z}_{T}\| \leq 1$ and ${T}= \mathcal{Z}_{T} (I-\mathcal{Z}_{T}^{*}\mathcal{Z}_{T})^{-\frac{1}{2}}$
		\item	$\big(\mathcal{Z}_{T}\big)^{*} = \mathcal{Z}_{T^{*}}$
		\item If $T$ is normal, then $\mathcal{Z}_{T}$ is normal.
	\end{enumerate}
\end{theorem}
%%%%%%%%%%%%%%%%%%%%%%%%%%%%%%%%%%%%%%%%%%%%%%%%%%%%%%%%%%%%%%%%%%%%%%%%%%%%%%%%%%%%%%%%%%%%%%%%%%%%%%%%%%%%%%%%%%%
\begin{theorem}\label{unboundedcartesian}
	Let $T \in \mathcal{C}(\mathcal{H})$ be normal. Then there exists an anti self-adjoint and unitary operator $J \in \mathcal{B}(\mathcal{H})$ such that $J$ commutes with $T$, that is $JT \subseteq TJ$.
\end{theorem}
\begin{proof}
	It is clear from the Theorem \ref{Ztransform},  that $\mathcal{Z}_{T}$ is a bounded right linear normal operator. By Proposition \cite[Theorem 5.9]{Ghiloni}, there exists an anti self-adjoint and unitary operator $J\in \mathcal{B}(\mathcal{H})$ such that
	$\mathcal{Z}_{T} J = J \mathcal{Z}_{T} \; \text{and} \; J\mathcal{Z}_{T}^{*}= \mathcal{Z}_{T}^{*}J$.
	This implies that $J(I- \mathcal{Z}_{T}^{*} \mathcal{Z}_{T}) = (I- \mathcal{Z}_{T}^{*} \mathcal{Z}_{T})J$. So $J$  commutes with the square root of bounded positive operator $I-\mathcal{Z}_{T}^{*}\mathcal{Z}_{T}$, that is $J (I-\mathcal{Z}_{T}^{*}\mathcal{Z}_{T})^{\frac{1}{2}} = (I-\mathcal{Z}_{T}^{*}\mathcal{Z}_{T})^{\frac{1}{2}} J$.
	
	Now we show that $J$ commutes with the inverse of $(I-\mathcal{Z}_{T}^{*}\mathcal{Z}_{T})^{\frac{1}{2}}$, which is an unbounded operator. Let $x \in \mathcal{D}((I-\mathcal{Z}_{T}^{*}\mathcal{Z}_{T})^{-\frac{1}{2}}) = R((I-\mathcal{Z}_{T}^{*}\mathcal{Z}_{T})^{\frac{1}{2}})$. Then $x = (I-\mathcal{Z}_{T}^{*}\mathcal{Z}_{T})^{\frac{1}{2}}y$, for some $y \in \mathcal{D}((I-\mathcal{Z}_{T}^{*}\mathcal{Z}_{T})^{\frac{1}{2}})$ and $Jx = J(I-\mathcal{Z}_{T}^{*}\mathcal{Z}_{T})^{\frac{1}{2}}y = (I-\mathcal{Z}_{T}^{*}\mathcal{Z}_{T})^{\frac{1}{2}}Jy \in R((I-\mathcal{Z}_{T}^{*}\mathcal{Z}_{T})^{\frac{1}{2}})$. This implies $Jx \in \mathcal{D}((I-\mathcal{Z}_{T}^{*}\mathcal{Z}_{T})^{-\frac{1}{2}})$. Moreover,
	\begin{equation*}
	J (I-\mathcal{Z}_{T}^{*}\mathcal{Z}_{T})^{-\frac{1}{2}}x= Jy = (I-\mathcal{Z}_{T}^{*}\mathcal{Z}_{T})^{-\frac{1}{2}}Jx.		
	\end{equation*}
	It is enough to show that $JT \subseteq TJ$. Since $T = \mathcal{Z}_{T}(I-\mathcal{Z}_{T}^{*}\mathcal{Z}_{T})^{-\frac{1}{2}}$ and $\mathcal{D}(T) = \mathcal{D}((I-\mathcal{Z}_{T}^{*}\mathcal{Z}_{T})^{-\frac{1}{2}})$ we see that $Jx \in \mathcal{D}(T)$,for every $x \in \mathcal{D}(T)$. Furthermore,
	\begin{align*}
	JTx = J\mathcal{Z}_{T}(I-\mathcal{Z}_{T}^{*}\mathcal{Z}_{T})^{-\frac{1}{2}}x
	&=\mathcal{Z}_{T} J (I-\mathcal{Z}_{T}^{*}\mathcal{Z}_{T})^{-\frac{1}{2}}x\\
	&= \mathcal{Z}_{T} (I-\mathcal{Z}_{T}^{*}\mathcal{Z}_{T})^{-\frac{1}{2}}Jx\\
	&=TJx.
	\end{align*}
	Hence the result.
\end{proof}
%%%%%%%%%%%%%%%%%%%%%%%%%%%%%%%%%%%%%%%%%%%%%%%%%%%%%%%%%%%%%%%%%%%%%%%%%%%
\begin{lemma}\label{Ztransformextension}
	Let  $J \in \mathcal{B}(\mathcal{H})$ be an anti self-adjoint and unitary. If $T \colon \mathcal{D}(T)\subset \mathcal{H}^{Jm}_{+} \to \mathcal{H}^{Jm}_{+}$ is $\mathbb{C}_{m}$- linear, then $\widetilde{\mathcal{Z}}_{T}= \mathcal{Z}_{\widetilde{T}}$.
\end{lemma}
\begin{proof} By Proposition \ref{extension}, we have
	\begin{align*}
	\widetilde{\mathcal{Z}}_{T} &= \widetilde{T} {(\widetilde{I}_{\mathcal{H}^{Jm}_{+}}+\widetilde{T}^{*}\widetilde{T})}^{-\frac{1}{2}}\\
	&= \widetilde{T}{(I_{\mathcal{H}}+ \widetilde{T}^{*}\widetilde{T})}^{-\frac{1}{2}}\\
	&= \mathcal{Z}_{\widetilde{T}}.
	\end{align*}
	It is enough to show that $\widetilde{(I_{\mathcal{H}^{Jm}_{+}}+ T^{*}T)} = (I_{\mathcal{H}} + \widetilde{T}^{*}\widetilde{T})$.
	
	Let $x \in \mathcal{D}(\widetilde{I_{\mathcal{H}^{Jm}_{+}}+ T^{*}T})$. Then $x = x_{1}+x_{2}$, where $x_{1} \in \mathcal{D}(I_{\mathcal{H}^{Jm}_{+}}+ T^{*}T)= \mathcal{D}(T^{*}T) $ and $x_{2} \in \Phi(\mathcal{D}(I_{\mathcal{H}^{Jm}_{+}}+ T^{*}T))= \Phi(\mathcal{D}(T^{*}T))$, we have
	\begin{align*}
	\widetilde{(I_{\mathcal{H}^{Jm}_{+}}+ T^{*}T)}(x) &= {(I_{\mathcal{H}^{Jm}_{+}}+ T^{*}T)}(x_{1}) - {(I_{\mathcal{H}^{Jm}_{+}}+ T^{*}T)}(x_{2} \cdot n) \cdot n \\
	&= (x_{1}+x_{2}) + \widetilde{T^{*}T}(x_{1}+x_{2})\\
	&=(I_{\mathcal{H}}+\widetilde{T}^{*}\widetilde{T} )(x). \qedhere
	\end{align*}
\end{proof}
%%%%%%%%%%%%%%%%%%%%%%%%%%%%%%%%%%%%%%%%%%%%%%%%%%%%%%%%%%%%%%%%%%%%%%%%%%%%%%%%%%%%%%%%%%%%%%%%%%%%%%%%%%%%%%%%%%%%%%%%%%

Next we prove that a linear operator on any $\mathbb{C}_{m}$- Hilbert space $K$ can be extended uniquely to a quaternionic linear operator on some quaternionic Hilbert space $\mathcal{H}$, which is associated to $K$. It is enough to prove that $K = \mathcal{H}^{Jm}_{+}$, for some anti self-adjoint and unitary $J\in \mathcal{B}(\mathcal{H})$, then the result follows from Proposition \ref{extension}

%%%%%%%%%%%%%%%%%%%%%%%%%%%%%%%%%%%%%%%%%%%%%%%%%%%%%%%%%%%%%%%%%%%%%%%%%%%%%%%%%%%%%%%%%%%%%%%%%%%%%%%%%%%%%%%%%%%%%%%%%%
\begin{lemma}\label{separable}
	Let $m, n \in \mathbb{S}$ with $mn = -nm$ and $J \in \mathcal{B}(\mathcal{H})$ be anti self-adjoint and unitary. Then $\mathcal{H}^{Jm}_{+}$ is separable if and only if $\mathcal{H}$ is separable.
\end{lemma}
\begin{proof}
	Suppose that $\mathcal{H}^{Jm}_{+}$ is separable. Let $D_{+}$ be countable dense subset of $\mathcal{H}^{Jm}_{+}$. Define
	\begin{equation*}
	D := \big\{a + b \cdot n :\; a,b \in D_{+}\big\}.
	\end{equation*}
	If $x \in \mathcal{H}$, then $x = a + b \cdot n$, for some $a,b \in \mathcal{H}^{Jm}_{+}$. Since $D_{+}$ is dense in $\mathcal{H}^{Jm}_{+}$, there exist $(a_{\ell}), (b_{\ell})$ in $D_{+}$ such that
	\begin{equation*}
	\|a_{\ell}-a\| \longrightarrow 0 \; \text{and}\; \|b_{\ell}-b\| \longrightarrow 0,\; \text{as}\; \ell \to \infty.
	\end{equation*}
	This implies that $(a_{\ell} + b_{\ell} \cdot n) \subset D$ and
	\begin{equation*}
	\|(a_{\ell}+b_{\ell} \cdot n) - (a+b\cdot n)\| \leq  \|a_{\ell}-a\| + \|b_{\ell}-b\| \longrightarrow 0, \; \text{as} \; \ell \to \infty.
	\end{equation*}
	Therefore $\mathcal{H}$ is separable.
	
	Assume that $\mathcal{H}$ is separable. Let $D \subset \mathcal{H}$ be a countable dense set. If $x \in \mathcal{H}$, then there exist $(x_{\ell}) \subseteq D$ such that
	\begin{equation*}
	\|x_{\ell} - x\| \longrightarrow 0, \; \text{as} \; \ell \to \infty.
	\end{equation*}
	Define $P_{+} \colon \mathcal{H} \to \mathcal{H}$ by
	\begin{equation*}
	P_{+}(x) = \frac{1}{2}(x-Jxm), \; \text{for all}\; x \in \mathcal{H}.
	\end{equation*}
	Clearly, $P_{+}^{2} = P_{+}$ and $R(P_{+})= \mathcal{H}^{Jm}_{+}$. 	Let $D_{+} : = \Big\{\frac{1}{2}(x-Jxm):\; x \in D\Big\}$. Then $D_{+}$ is countable subset of $\mathcal{H}^{Jm}_{+}$. It is enough to show $D_{+}$ is dense in $\mathcal{H}^{Jm}_{+}$. If $y \in \mathcal{H}^{Jm}_{+}$, then there exist some $x \in \mathcal{H}$ such that $y = P_{+}(x)= \frac{1}{2}(x - Jxm)$. This implies that
	\begin{align*}
	\Big\|\frac{1}{2}\big(x_{\ell} - Jx_{\ell}m\big) - \frac{1}{2}\big(x-Jxm\big)\Big\| &= \frac{1}{2} \Big\|(x_{\ell}-x) - \big(Jx_{\ell}m -Jxm\big)\Big\|\\
	&\leq \frac{1}{2} \big\|x_{\ell}-x\big\| + \frac{1}{2}\big\|J(x_{\ell}-x)m\big\|\\
	&= \big\|x_{\ell}-x\big\| \\
	& \longrightarrow 0, \; \text{as} \; \; \ell \to \infty.
	\end{align*}
	Hence $\mathcal{H}^{Jm}_{+}$ is separable.
\end{proof}

%%%%%%%%%%%%%%%%%%%%%%%%%%%%%%%%%%%%%%%%%%%%%%%%%%%%%%%%%%%%%%%%%%%%%%%%%%%%%%%%%%%%%%%%%%%%%%%%%%%%%%%%%%%%%%%%%%%%%%%%%%%%%%%%%%%%%%%%%%%%%%%%%%%%%%%%%%
%\begin{lemma}\label{innerproductsum}\cite[Proposition 3.11(a)]{Ghiloni}
%	If $x \in \mathcal{H}^{Jm}_{+}$ and $y \in \mathcal{H}^{Jm}_{-}$, then $ \left\langle x|y\right\rangle + \left\langle y|x\right\rangle=0 $.
%\end{lemma}
%%%%%%%%%%%%%%%%%%%%%%%%%%%%%%%%%%%%%%%%%%%%%%%%%%%%%%%%%%%%%%%%
%We recall that every $\mathbb{C}_{m}$- linear operator in $\mathcal{H}^{Jm}_{+}$ can be extended uniquely to a right linear operator in $\mathcal{H}$. Also the converse is possible with some assumptions, which are given below.
%%%%%%%%%%%%%%%%%%%%%%%%%%%%%%%%%%%%%%%%%%%%%%%%%%%%%%%%%%%%%%%%%%%%%%%%%%%%%%%%%%%%%%%%%%%%%%%%%%%%%%%%%%%%%%%%%%%%%%%%%%%
\begin{remark}
	Let $q \in \mathbb{H}$ and $T \colon \mathcal{D}(T) \subseteq \mathcal{H}^{Jm}_{+} \to \mathcal{H}^{Jm}_{+}$ be a $\mathbb{C}_{m}$- linear. Then by  Proposition \ref{extension}(\ref{extnmulti}), we have
	\begin{equation}\label{deltaextn}
	\Delta_{q}(\widetilde{T}) = \widetilde{\Delta}_{q}(T),
	\end{equation}
	where $\widetilde{\Delta}_{q}(T)$ denotes the extension of $\Delta_{q}(T)$ to $\mathcal{H}$.
\end{remark}
%%%%%%%%%%%%%%%%%%%%%%%%%%%%%%%%%%%%%%%%%%%%%%%%%%%%%%%%%%%%%%%%%%%%%%%%%%%%%
%The Proposition \ref{extension} shows that for any $m\in \mathbb{S}$, the extension is possible only for the operators on slice Hilbert space of $\mathcal{H}$. Whereas, we show that the extension is also possible for the operators in any $\mathbb{C}_{m}$- Hilbert space. We establish this result by showing that every $\mathbb{C}_{m}$- Hilbert space is a slice Hilbert space of some quaternionic Hilbert space.
%%%%%%%%%%%%%%%%%%%%%%%%%%%%%%%%%%%%%%%%%%%%%%%%%%%%%%%%%%%%%%%%%%%%%%%%%%%%%%%%%%%%%%%%%%%%%%%%%%%%%%%%%%%%%%%%%%%%%%%%%%%%%
\begin{proposition}\label{slice}
	Let $m \in \mathbb{S}$. If $K$ is a $\mathbb{C}_{m}$ - Hilbert space, then $\mathcal{H} = K \times K$ can be given a quaternionic Hilbert space structure and  there exist an anti self-adjoint and unitary $J \in \mathcal{B}(\mathcal{H})$ such that
	\begin{equation*}
	K = \mathcal{H}^{Jm}_{+}.
	\end{equation*}
\end{proposition}
\begin{proof}
	Let $ \mathcal{H}:= K \times K$. We define addition and scalar multiplication on $\mathcal{H}$ as follows:
	\begin{equation*}
	(x,y) + (z, w) := (x+z, y+w), \; \text{for all}\; (x,y), (z,w) \in \mathcal{H}.
	\end{equation*}
	Let $q \in \mathbb{H}$. Then $q = \alpha + \beta \cdot n $, for some $\alpha, \beta \in \mathbb{C}_{m}$, where $n \in \mathbb{S}$ is such that $ m\cdot n = - n \cdot m $. Define a right scalar multiplication by
	\begin{equation}\label{slicecartesian}
	(x,y) \cdot (\alpha + \beta \cdot n) := (x \cdot \alpha - y \cdot \beta, \ x \cdot \beta - y \cdot \alpha).
	\end{equation}
	Define $ \left\langle \cdot | \cdot \right\rangle \colon \mathcal{H} \times \mathcal{H} \to \mathbb{H}$ by
	\begin{equation}\label{defofinnerproduct}
	\left\langle (x,y) | (z,w) \right\rangle = [ \left\langle x|z\right\rangle_{K} + \left\langle w | y \right\rangle_{K}] + [ \left\langle x | w\right\rangle_{K} - \left\langle z|y \right\rangle_{K} ] \cdot n,
	\end{equation}
	where $\left\langle \cdot | \cdot \right\rangle_{K}$ is the inner product on $K$.  Let $(x_{1},y_{1}), (x_{2}, y_{2}), (x_{3},y_{3}) \in \mathcal{H} $ and $q \in \mathbb{H}$. Then the following hold:  	
	\begin{enumerate}
		\item  Equation (\ref{defofinnerproduct}) implies that
		\begin{align*}
		\left\langle (x_{1},y_{1}) | (x_{1},y_{1})\right\rangle &= \Big[\left\langle x_{1} | x_{1}\right\rangle_{K} + \left\langle y_{1}| y_{1}\right\rangle_{K} \Big] + \Big[\left\langle x_{1} | y_{1}\right\rangle_{K} - \left\langle x_{1} | y_{1}\right\rangle_{K}\Big] \cdot n\\
		&= \|x_{1}\|^{2}+ \|y_{1}\|^{2} \\
		& \geq 0.
		\end{align*}
		Moreover,
		\begin{equation*}
		\left\langle (x_{1}, y_{1}) | (x_{1}, y_{1})\right\rangle = 0  \Leftrightarrow \|x_{1}\|^{2} + \|y_{1}\|^{2} = 0 \\
		\Leftrightarrow (x_{1},y_{1}) = (0,0).
		\end{equation*}
		\item If $q = \alpha + \beta \cdot n $, for $\alpha,\beta \in \mathbb{C}_{m}$ then
		\begin{align*}
		\left\langle (x_{1},y_{1})| (x_{2},y_{2})+(x_{3}, y_{3})\cdot q\right\rangle &= \left\langle (x_{1},y_{1})| (x_{2},y_{2})+(x_{3}, y_{3})\cdot (\alpha + \beta \cdot n)\right\rangle  \\
		&= \left\langle (x_{1},y_{1}) | (x_{2}+x_{3}\alpha -y_{3}\cdot \beta, y_{2}+x_{3} \cdot \beta-y_{3} \cdot \alpha) \right\rangle\\
		&= \left\langle x_{1} | x_{2}+ x_{3} \alpha-y_{3}\beta\right\rangle_{K} + \left\langle y_{2}+x_{3}\beta-y_{3}\alpha | y_{1}\right\rangle_{K} \\
		& \;\; + [\left\langle x_{1} | y_{2}+x_{3}\beta-y_{3}\alpha \right\rangle_{K} - \left\langle x_{2}+x_{3}\alpha-y_{3}\beta | y_{1}\right\rangle_{K} ] \cdot n \\
		&= \left\langle (x_{1},y_{1}) | (x_{2},y_{2})\right\rangle + \left\langle (x_{1},y_{1}) | (x_{3},y_{3})\right\rangle \cdot (\alpha + \beta \cdot n)\\
		&= \left\langle (x_{1},y_{1}) | (x_{2},y_{2})\right\rangle + \left\langle (x_{1},y_{1}) | (x_{3},y_{3})\right\rangle \cdot q.
		\end{align*}
		\item Conjugate property:
		\begin{align*}
		\left\langle (x_{1},y_{1}) | (x_{2},y_{2}) \right\rangle &= \Big[ \left\langle x_{1}|x_{2}\right\rangle_{K} + \left\langle y_{2} | y_{1} \right\rangle_{K}\Big] + \Big[ \left\langle x_{1} | y_{2}\right\rangle_{K} - \left\langle x_{2}|y_{1} \right\rangle_{K} \Big] \cdot n \\
		&= \overline{\left\langle x_{2}|x_{1}\right\rangle_{K} + \left\langle y_{1}|y_{2}\right\rangle_{K}} + \overline{[\left\langle y_{2}|x_{1}\right\rangle_{K} - \left\langle y_{1}|x_{2}\right\rangle_{K}]} \cdot n\\
		& = \overline{\left\langle (x_{2},y_{2}) | (x_{1},y_{1})\right\rangle}.
		\end{align*}
	\end{enumerate}
	This implies that $\left\langle \cdot | \cdot \right\rangle$ is an inner product on $\mathcal{H}$. The induced norm on $\mathcal{H}$ is given by
	\begin{equation}\label{defofinducednorm}
	\|(x,y)\|^{2} = \left\langle (x,y)|(x,y)\right\rangle = \|x\|^{2}_{K} + \|y\|^{2}_{K},
	\end{equation} 	
	for all $(x,y) \in \mathcal{H}$. Here $\|\cdot\|_{K}$ denote the norm on $K$ induced from $\left\langle \cdot | \cdot \right\rangle_{K}$.
	
	Since $K$ is complete, we see that $\mathcal{H}$ is complete with respect to the norm defined in Equation (\ref{defofinducednorm}).  Therefore $\mathcal{H}$ is a right quaternionic Hilbert space.
	
	If we identify $x$ in $K$ by $(x,0)$ then $x+y\cdot n$ is identified by $(x,y)$. That is
	\begin{equation*}
	x+y\cdot n = (x,0) + (y,0) \cdot n = (x,0) + (0,y) =  (x,y).
	\end{equation*}
	Here we used Equation (\ref{slicecartesian}) to conclude $(0,y) = (y,0) \cdot n$.
	From now onwards we write $x + y \cdot n$ instead of $(x,y) \in \mathcal{H}$.
	
	Define $J \colon \mathcal{H} \to \mathcal{H}$ by
	\begin{equation*}
	J(x+y \cdot n) = (x-y\cdot n) \cdot m, \; \text{for all}\; x+y\cdot n \in \mathcal{H}.
	\end{equation*}
	We shall prove that $J$ is anti self-adjoint and unitary. Let $x = x_{+} + x_{-} \in \mathcal{H}^{Jm}_{+} \oplus \mathcal{H}^{Jm}_{-}$. Then
	\begin{align*}
	\left\langle x\;\big|\;Jy\right\rangle &= \big\langle x_{+}+x_{-}\cdot n \;\big|\; J(y_{+}+y_{-} \cdot n)\big\rangle\\
	&= \big\langle x_{+} + x_{-} \cdot n \;\big|\; (y_{+}- y_{-} \cdot n) \cdot m\big\rangle \\
	&= \left\langle x_{+} \;\big|\; (y_{+}- y_{-} \cdot n) \cdot m \right\rangle + \overline{n}\left\langle x_{-} \;\big|\; (y_{+}- y_{-} \cdot n ) \cdot m \right\rangle\\
	&= \left\langle x_{+} \cdot \overline{m} \;\big|\; y_{+}+y_{-} \cdot n \right\rangle + \left\langle x_{-} \cdot \overline{m} \cdot n \;\big|\; y_{+}+y_{-} \cdot n \right\rangle \\
	&= \left\langle (-x_{+} + x_{-} \cdot n) \cdot m \;\big|\; y_{+}+y_{-} \cdot n\right\rangle\\
	&= \left\langle (-x_{+} + x_{-} \cdot n) \cdot m \;\big|\; y \cdot n\right\rangle.
	\end{align*}
	This implies that
	\begin{equation*}
	J^{*}(u) = J^{*}(u_{+}+u_{-}\cdot n) = (-u_{+}+u_{-}\cdot n) \cdot m = - (u_{+}-u_{-}\cdot n) = - J(u), \; \text{for all}\; u \in \mathcal{H}.
	\end{equation*} Therefore $J^{*}= -J$ and $J^{*}J = JJ^{*}=I$.
	
	We claim that $\mathcal{H}^{Jm}_{+}=K$. If $x \in \mathcal{H}^{Jm}_{+}$, then $J(x) = J(x_{+}+x_{-}\cdot n) = (x_{+}+x_{-}\cdot n )\cdot m$.  By the definition of $J$, we have
	\begin{equation*}
	(x-y\cdot n)\cdot m = (x + y\cdot n)\cdot m.
	\end{equation*}
	It implies that $y = 0$, that is $x \in K$. Conversely, if $x \in K$ then $J(x) = x \cdot m$, it shows that $x \in \mathcal{H}^{Jm}_{+}$. Hence $K = \mathcal{H}^{Jm}_{+}$.
\end{proof}
%%%%%%%%%%%%%%%%%%%%%%%%%%%%%%%%%%%%%%%%%%%%%%%%%%%%%%%%%%%%%%%%%%%%%%%%%%%%%%%%%%%%%%%%%%%%%%%%%%%%%%%%%%%%%%%
Now it is clear from Proposition \ref{slice}  that  a linear operator on any $\mathbb{C}_{m}$- Hilbert space can be extended uniquely to a quaternionic linear operator on some quaternionic Hilbert space.
%%%%%%%%%%%%%%%%%%%%%%%%%%%%%%%%%%%%%%%%%%%%%%%%%%%%%%%%%%%%%%%%%%%%%%%%%%%%%%%%%%%%%%%%%%%%%%%%%%%%%%%%%

The spherical spectrum of $\widetilde{T}$, for $T\in \mathcal{B}(\mathcal{H}^{Jm}_{+})$, is given  as follows:
\begin{lemma}\label{extnspectrum}\cite[Proposition 5.11]{Ghiloni}.
	Let $J \in \mathcal{B}(\mathcal{H})$ be anti self-adjoint and unitary, $m \in \mathbb{S}$.
	Let $T \in \mathcal{B}(\mathcal{H}^{Jm}_{+})$ and  $\widetilde{T}$ be the extension of $T$ as given in Proposition \ref{extension}. Then
	\[\sigma_{S}(\widetilde{T}) = \bigcup\limits_{\lambda \;\in \; \sigma(T)} [\lambda].\]
\end{lemma}
%%%%%%%%%%%%%%%%%%%%%%%%%%%%%%%%%%%%%%%%%%%%%%%%%%%%%%%%%%%%%%%%%%%%%%%%%%%%%%%%%%%%%%%%%%%%%%%%%%%%%%%%%%%%
\begin{theorem} \label{qspintermscsp} \cite[Corollary 5.13]{Ghiloni}
	Let $T \in \mathcal{B}(\mathcal{H})$ be normal and $J \in \mathcal{B}(\mathcal{H})$ be anti self-adjoint and unitary such that $TJ = JT$ and let $m \in \mathbb{S}$. Then the following holds true:
	\begin{equation*}
	\sigma(T|_{\mathcal{H}^{Jm}_{+}}) = \sigma_{S}(T) \cap \mathbb{C}_{m}^{+}, \; \sigma(T|_{\mathcal{H}^{Jm}_{-}}) = \sigma_{S}(T) \cap \mathbb{C}_{m}^{-}.
	\end{equation*}
	and  hence
	\begin{equation*}
	\sigma(T|_{\mathcal{H}^{Jm}_{+}}) = \overline{\sigma(T|_{\mathcal{H}^{Jm}_{-}})}.
	\end{equation*}
\end{theorem}
Next, we generalize Proposition \ref{extension} for linear operators between two  $\mathbb{C}_{m}$ - Hilbert spaces.

\begin{theorem}\label{extension1}
	Let $\mathcal{H}_{1}, \mathcal{H}_{2}$ be quaternionic Hilbert spaces. Let $J_{1}, J_{2}$ be anti self-adjoint unitary operators on $\mathcal{H}_{1}$ and $\mathcal{H}_{2}$, respectively.  If $m \in \mathbb{S}$ and $ T \colon \mathcal{D}(T) \subseteq {\mathcal{H}_{1}}^{J_{1}m}_{+} \to {\mathcal{H}_{2}}^{J_{2}m}_{+}$ is a $\mathbb{C}_{m}$- linear, then there exists unique right $\mathbb{H}$- linear operator $\widetilde{T} \colon \mathcal{D}(\widetilde{T}) \to \mathcal{H}_{2} $ such that
	$ \mathcal{D}(\widetilde{T}) \bigcap {\mathcal{H}_{1}}^{J_{1}m}_{+} = \mathcal{D}(T), \ J_1(\mathcal{D}(\widetilde{T})) \subset \mathcal{D}(\widetilde{T})$ and $ \widetilde{T}(x) = T(x), $ for every $x \in \mathcal{D}(T).$
	Furthermore,
	\begin{enumerate}
		\item If $T \colon {\mathcal{H}_{1}}^{J_{1}m}_{+} \to {\mathcal{H}_{2}}^{J_{2}m}_{+}$ is bounded, then $\widetilde{T} \in \mathcal{B}(\mathcal{H}_{1}, \mathcal{H}_{2}) $ and  \label{norm1}$\|\widetilde{T}\| = \|T\|$.
		\item \label{commute1}$J_{2}\widetilde{T} = \widetilde{T} J_{1}$ on $\mathcal{D}(\widetilde{T})$.
	\end{enumerate}
	On the other hand, let $V \colon  \mathcal{D}(V) \subseteq \mathcal{H}_{1} \to \mathcal{H}_{2}$ be a right $\mathbb{H}$- linear . Then $V = \widetilde{U}$, for a unique $U \colon \mathcal{D}(V)\cap {\mathcal{H}_{1}}^{J_{1}m}_{+} \to {\mathcal{H}_{2}}^{J_{2}m}_{+}$ if and only if $J_{1}(\mathcal{D}(V)) \subseteq \mathcal{D}(V)$ and $J_{2}V(x) = VJ_{1}(x)$,\; for all $x \in \mathcal{D}(V)$.
\end{theorem}
\begin{proof} First we show that the extension  is unique. Let $n \in \mathbb{S}$ be such that $m\cdot n = - n \cdot m $. Then $q \in \mathbb{H}$ can be written by
	\begin{equation*}
	q = q_{0}+q_{1}m+q_{2}n+q_{3}mn,
	\end{equation*}
	where $q_{\ell} \in \mathbb{R}$ for $\ell = 0,1,2,3$.
	
	Define $\Phi \colon \mathcal{H}_{1} \to \mathcal{H}_{1}$ by
	\begin{equation} \label{phi}
	\Phi(x) = x \cdot n, \; \text{for all} \;  x \in \mathcal{H}_{1}.
	\end{equation}
	It is clear that $\Phi$ is anti  $\mathbb{C}_{m}$- linear isomorphism. Moreover, $\Phi({\mathcal{H}_{1}}^{J_{1}m}_{\pm}) ={\mathcal{H}_{1}}^{J_{1}m}_{\mp}$.
	
	Assume that there exists an extension $\widetilde{T}$ of $T$ such that $\widetilde{T}(x) = Tx$, for all $x \in \mathcal{D}(T)$, with $\mathcal{D}(\widetilde{T}) \bigcap {\mathcal{H}_{1}}^{J_{1}m}_{+} = \mathcal{D}(T) $ and $J_1(\mathcal{D}(\widetilde{T})) \subset \mathcal{D}(\widetilde{T})$. We show that   $\mathcal{D}(\widetilde{T}) \cap {\mathcal{H}_{1}}^{J_{1}m}_{-} = \Phi(\mathcal{D}(T))$. Suppose $x \in \mathcal{D}(\widetilde{T}) \bigcap {\mathcal{H}_{1}}^{J_{1}m}_{-}$, then $x = \Phi(y)$, for some $y \in {\mathcal{H}_{1}}^{J_{1}m}_{+}$. Equivalently,  $y = - \Phi(x)$. Since $\Phi(\mathcal{D}(\widetilde{T})) = \mathcal{D}(\widetilde{T})$, we have $y \in \mathcal{D}(\widetilde{T}) \bigcap {\mathcal{H}_{1}}^{J_{1}m}_{+} = \mathcal{D}(T)$. It implies that $x \in \Phi(\mathcal{D}(T))$. If  $x \in \Phi(\mathcal{D}(T))$, then $x = \Phi(y)$, for some $y \in \mathcal{D}(T) = \mathcal{D}(\widetilde{T}) \bigcap {\mathcal{H}_{1}}^{J_{1}m}_{+}$. Therefore $x \in \mathcal{D}(\widetilde{T}) \bigcap {\mathcal{H}_{1}}^{J_{1}m}_{-}$.
	
	Let $x \in \mathcal{D}(\widetilde{T})$. Then
	\begin{equation*}
	x_{1}:= \frac{x - J_{1}xm}{2} \in {\mathcal{H}_{1}}^{J_{1}m}_{+}\; \; ; \; \;
	x_{2}:= \frac{x+J_{1}xm}{2} \in {\mathcal{H}_{1}}^{J_{1}m}_{-}.
	\end{equation*}
	Moreover,
	\begin{equation*}
	x = x_{1} + x_{2}.
	\end{equation*}
	This implies $\mathcal{D}(\widetilde{T}) = \mathcal{D}(T) \oplus \Phi(\mathcal{D}(T))$.  By the assumption on $\widetilde{T}$, we have
	\begin{equation} \label{defoftilde}
	\widetilde{T}(x) = T(x_{1}) - T(x_{2} \cdot n) \cdot n.
	\end{equation}
	The definition of $\widetilde{T}$ in Equation (\ref{defoftilde}) is determined by $T$. Hence the extension is unique.
	
	To show the  existence, define $\mathcal{D}(\widetilde{T}):=\mathcal{D}(T) \oplus \Phi(\mathcal{D}(T))$. It implies that $\mathcal{D}(\widetilde{T}) \bigcap {\mathcal{H}_{1}}^{J_{1}m}_{+} = \mathcal{D}(T)$. Since $\mathcal{D}(\widetilde{T})$ is right $\mathbb{H}$- linear subspace of $\mathcal{H}_{1}$, we have
	\begin{equation*}
	J_{1}(x) = J_{1}(x_{1}) + J_{1}(x_{2}) = x_{1}m - x_{2}m \in \mathcal{D}(\widetilde{T}).
	\end{equation*}
	In fact by Equation (\ref{defoftilde}), we have $\widetilde{T}(x) = T(x)$, for all $x \in \mathcal{D}(T)$.

	Proof of (\ref{norm1}): If $T \colon {\mathcal{H}_{1}}^{J_{1}m}_{+} \to {\mathcal{H}_{2}}^{J_{2}m}_{+}$ is bounded, then $\mathcal{D}(\widetilde{T}) = {\mathcal{H}_{1}}^{J_{1}m}_{+} \oplus {\mathcal{H}_{1}}^{J_{1}m}_{-} = \mathcal{H}_{1}$. Since $\widetilde{T}$ is the extension of $T$, it follows that  $\|T\| \leq \|\widetilde{T}\|$. If $ x = x_{1}+x_{2} \in \mathcal{H}_{1}$, then we have
	\begin{align*}
	\|\widetilde{T}x\|^{2} = \|{T}(x_{1}) - {T}(x_{2}\cdot n)\|^{2} &= \|Tx_{1}\|^{2} + \|T(x_{2}\cdot n)\|^{2}\\
	&\leq \|T\|^{2} (\|x_{1}\|^{2}+\|x_{2}\|^{2})\\
	&\leq \|T\|^{2} \|x\|^{2}.
	\end{align*}
	This implies that $\|\widetilde{T}\|\leq \|T\|$. Hence $\|\widetilde{T}\| = \|T\|$.
	
	\noindent Proof of (\ref{commute1}): If $x \in \mathcal{D}(\widetilde{T})$,  then $x = x_{1}+x_{2}$ with $x_{1} \in \mathcal{D}(T), \; x_{2} \in \Phi(\mathcal{D}(T))$ and $J_{1}(x) \in \mathcal{D}(\widetilde{T})$. Moreover,
	\begin{align*}
	J_{2}\widetilde{T}(x_{1}+x_{2}) &= J_{2}[T(x_{1})- T(x_{2}\cdot n)\cdot n]  \\
	&= J_{2}(T(x_{1})) - J_{2}(T(x_{2}\cdot n)) \cdot n \\
	&= T(x_{1}) \cdot m - T(x_{2}\cdot n) \cdot m \cdot n \\
	&= T(x_{1}\cdot m) - T(-x_{2}\cdot m\cdot n)\cdot n  \\
	&= T(J_{1}x)-T(J_{1}x_{2}\cdot n) \cdot n \\
	&= \widetilde{T}(J_{1}x_{1}+J_{1}x_{2})\\
	&= \widetilde{T}J_{1}(x).
	\end{align*}
	If $V = \widetilde{U}$, for some $U \colon \mathcal{D}(V)\bigcap {\mathcal{H}_{1}}^{J_{1}m}_{+} \to {\mathcal{H}_{2}}^{J_{2}m}_{+}$, then $J_{1}(\mathcal{D}(V)) \subset \mathcal{D}(V)$. It is clear from the earlier proof that $\mathcal{D}(V) = \mathcal{D}(U) \oplus \Phi(\mathcal{D}(U))$. For $x=x_{1}+x_{2} \in \mathcal{D}(V)$, we have
	\begin{align*}
	J_{2}V(x) &= J_{2}(Ux_{1} - U(x_{2} \cdot n) \cdot n) \\
	&= U(x_{1}) \cdot m - U(x_{2} \cdot n) \cdot m \cdot n\\
	&= U(x_{1} \cdot m) + U(x_{2} \cdot n \cdot m) \cdot n \\
	&= U(J_{1}x_{1}) - U(J_{1}x_{2} \cdot n ) \cdot n\\
	&= VJ_{1}(x).
	\end{align*}
	
	Conversely, assume that $J_{2}V = VJ_{1}.$  That is $J_{2}Vx = VJ_{1}x, \forall \ x \in \mathcal{D}(V)$. It implies that $J_{1}x \in \mathcal{D}(V)$, for every $x \in \mathcal{D}(V)$. Hence $V\big(\mathcal{D}(V) \bigcap {\mathcal{H}_{1}}^{J_{1}m}_{+}\big) \subseteq {\mathcal{H}_{2}}^{J_{2}m}_{+}$.
	
	Define $U\colon \mathcal{D}(V)\bigcap {\mathcal{H}_{1}}^{J_{1}m}_{+} \to {\mathcal{H}_{2}}^{J_{2}m}_{+}$ by \begin{equation*}
	Ux = Vx, \; \text{for all} \; x \in \mathcal{D}(U).
	\end{equation*}
	Here $V$ is right $\mathbb{H}$- linear extension of $U$ such that  $J_{1}(\mathcal{D}(V)) \subset \mathcal{D}(V)$, by the uniqueness of extension, we have $V = \widetilde{U}$.
\end{proof}
%%%%%%%%%%%%%%%%%%%%%%%%%%%%%%%%%%%%%%%%%%%%%%%%%%%%%%%%%%%%%%%%%%%%%%%%%%%%%
Our aim is to prove that $L^{2}(\Omega;\mathbb{C}_{m};\mu)= L^{2}(\Omega;\mathbb{H};\mu)^{Jm}_{+}$, for some anti self-adjoint and unitary $J \in \mathcal{B}(L^{2}\big(\Omega;\mathbb{H};\mu)\big)$. To establish this result, we need the following theorem.
%%%%%%%%%%%%%%%%%%%%%%%%%%%%%%%%%%%%%%%%%%%%%%%%%%%%%%%%%%%%%%%%%%%%%%%%%%%%%
\begin{theorem}\label{directsum}
	Let $m,n \in \mathbb{S}$ be such that $m \cdot n = - n \cdot m$. Let $(\Omega, \mu)$ be a measure space. Then $L^{2}(\Omega; \mathbb{C}_{m}; \mu)$ is closed in $L^{2}(\Omega; \mathbb{H}; \mu)$. Moreover,
	\begin{equation*}
	L^{2}(\Omega; \mathbb{H}; \mu) = L^{2}(\Omega; \mathbb{C}_{m}; \mu) \oplus \Phi(L^{2}(\Omega; \mathbb{C}_{m}; \mu)),
	\end{equation*} 	
	where $\Phi(f) = f\cdot n$, for every $f \in L^{2}(\Omega; \mathbb{H}; \mu)$.
\end{theorem}
\begin{proof} If $f \in L^{2}(\Omega; \mathbb{H}; \mu)$, then for all $x \in \Omega$, $f(x) = F_{1}(x) + F_{2}(x) \cdot n$, where
	\begin{align*}
	F_{1}(x) &= \frac{1}{2}(f(x)+\overline{f(x)}) - \frac{1}{2}(f(x)m+\overline{f(x)m})m ;\\
	F_{2}(x)&=\frac{1}{2}(f(x)n+\overline{f(x)n}) - \frac{1}{2}(f(x)mn+\overline{f(x)mn})m.
	\end{align*}
	Clearly, $F_{\ell}$ is $\mathbb{C}_{m}$- valued function on $\Omega$, for $\ell \in \{1,2\}$. Since
	\[\int\limits_{\Omega}|F_{\ell}(x)|^{2}d\mu(x) \leq \int\limits_{\Omega}|f(x)|^{2}d\mu(x) < \infty ,\]
	we conclude that $F_{\ell} \in L^{2}(\Omega;\mathbb{C}_{m}; \mu)$. This implies that $f = F_{1} + \Phi(F_{2})$. Thus
%	 by using the fact that  $L^{2}(\Omega; \mathbb{C}_{m}; \mu) \bigcap \Phi(L^{2}(\Omega; \mathbb{C}_{m}; \mu)) = \{0\}$, we have
	\begin{equation*}
	L^{2}(\Omega; \mathbb{H}; \mu) = L^{2}(\Omega; \mathbb{C}_{m}; \mu) \oplus \Phi(L^{2}(\Omega; \mathbb{C}_{m}; \mu)).
	\end{equation*}
	Now we  show that $L^{2}(\Omega; \mathbb{C}_{m}; \mu)$ is closed. Let $\{f_{k}\}$ be a sequence in $L^{2}(\Omega; \mathbb{H}; \mu)$. If $\{f_{k}\}$ converges to $f = F_{1}+F_{2} \cdot n$ in $L^{2}(\Omega; \mathbb{H}; \mu)$, then
	\begin{align*}
	\|f_{k}-f\|^{2} = \|(f_{k}-F_{1}) - F_{2}\cdot n\|^{2} = \|f_{k}-F_{1}\|^{2} + \|F_{2}\|^{2}.
	\end{align*}
	Since $f_{k} \to f$, it follows that $\|f_{k}-f_{1}\|^{2} \to 0$, as $k \to \infty$ and $f_{2} = 0$. Therefore $L^{2}(\Omega; \mathbb{C}_{m}; \mu)$ is closed in $L^{2}(\Omega; \mathbb{H}; \mu)$.
\end{proof}
%%%%%%%%%%%%%%%%%%%%%%%%%%%%%%%%%%%%%%%%%%%%%%%%%%%%%%%%%%%%%%%%%%%%%%%%%%%%%%%%%%%%%%%%%%%%%%%%%%%%%%%%%%%%%%%%%%%%%%%%%%%%%%%%%%%%%%%%%%%
\begin{corollary} Let $(\Omega,\mu)$ be a measure space and $m \in \mathbb{S}$. Then there exist an anti self-adjoint and unitary $J \in \mathcal{B}(L^{2}(\Omega; \mathbb{H}; \mu))$ such that
	\begin{equation*}
	L^{2}(\Omega; \mathbb{H}; \mu)^{Jm}_{+} = L^{2}(\Omega; \mathbb{C}_{m}; \mu).
	\end{equation*}
	\begin{proof} Let $f \in L^{2}(\Omega; \mathbb{H}; \mu)$. Then by Theorem \ref{directsum}, we can write $f= F_{1} + F_{2} \cdot n$, for some $F_{1}, F_{2} \in L^{2}(\Omega; \mathbb{C}_{m}; \mu)$. Define $J$ on $L^{2}(\Omega; \mathbb{H}; \mu)$ by
		\begin{equation*}
		J(F_{1}+F_{2}\cdot n) = (F_{1} - F_{2} \cdot n) \cdot m.
		\end{equation*}
		As in the Proof of the Proposition \ref{slice}, we can show that $J$ is anti self-adjoint and unitary, and 	
		\begin{equation*}
		L^{2}(\Omega; \mathbb{H}; \mu)^{Jm}_{+} = L^{2}(\Omega; \mathbb{C}_{m}; \mu). \qedhere
		\end{equation*}
	\end{proof}
\end{corollary}

\section{Spectral theorem: Bounded operators}
%%%%%%%%%%%%%%%%%%%%%%%%%%%%%%%%%%%%%%%%%%%%%%%%%%%%%%%%%%%%%%%%%%%%%%%%%%%%
In this section we prove the spectral theorem (multiplication form) for bounded normal operators on a quaternionic Hilbert space. We recall the spectral theorem in complex Hilbert spaces.
\begin{theorem}\cite[Theorem 11.5]{conway1}\label{complexbound}
	Let $K$ be a complex Hilbert space. If $N \in \mathcal{B}(K)$ is a normal operator, then there is a measure space $(X,\mu)$ and an essentially bounded $\mu$- measurable function $\phi \colon X \to \mathbb{C}$ such that $N$ is unitarily equivalent to $L_{\phi}$, where $L_{\phi}$ is a left multiplication by $\phi$  acting on $L^{2}(X;\mathbb{C};\mu)$.
	
	More over, if $K$ is separable then the measure space  obtained above $(X, \mu)$ is $\sigma$- finite.
\end{theorem}
Though the multiplication form of a bounded quaternionic normal operator is proved via integral representation (see \cite[Theorem 4.4]{Ghiloni2}), we prove this result similar to the classical setup  by exploiting Theorem \ref{complexbound} and Proposition \ref{extension}.
%%%%%%%%%%%%%%%%%%%%%%%%%%%%%%%%%%%%%%%%%%%%%%%%%%%%%%%%%%%%%%%%%
\subsection*{Multiplication form:}
\begin{theorem}\label{multiplication}\label{quaternionbounded}
	Let $T\in\mathcal{B}(\mathcal{H})$ be normal and fix $ m \in \mathbb{S}.$  Then there exists
	\begin{enumerate}
		\item[(a)] a Hilbert basis $\mathcal{N}_{m}$ of $\mathcal{H}$
		\item[(b)] a  measure space $(\Omega, \mu)$
		\item[(c)] a unitary operator $U \colon \mathcal{H} \to L^{2}(\Omega; \mathbb{H}; \mu)$ and
		\item[(d)]  an essentially bounded $\mu$- measurable function $ \phi \colon \Omega \to \mathbb{C}_{m}$
	\end{enumerate}
	such that, if $T$ is expressed with respect to $\mathcal{N}_{m}$, then
	\begin{equation*}
	T = U^{*}{{M}}_{\phi}U,
	\end{equation*}
	where ${M}_{\phi}$  is a bounded multiplication operator on $L^{2}(\Omega; \mathbb{H}; \mu)$.
	
	Moreover,
	\begin{enumerate}
		
		\item \label{norm} $\|T\| = \text{ess\ sup}(|\phi|)$ \\
		\item \label{spectrum}$\sigma_{S}(T) = \bigcup\limits_{\lambda \in \; \text{ess ran}\; (\phi)} [\lambda]$.
	\end{enumerate}
	Further more, if $\mathcal{H}$ is separable Hilbert space, then the obtained measure space $(\Omega, \mu)$ is $\sigma$- finite.
\end{theorem}
\begin{proof}
	By Remark \ref{extensionrmk}, $T_{+} \colon \mathcal{H}^{Jm}_{+} \to \mathcal{H}^{Jm}_{+}$ is the unique $\mathbb{C}_{m}$- linear bounded normal operator such that $\widetilde{T}_{+} = T$. If $\mathcal{N}_{m}$ be Hilbert basis for $\mathcal{H}^{Jm}_{+}$, then by Remark \ref{J}, $\mathcal{N}_{m}$ is Hilbert basis for $\mathcal{H}$ and moreover,
	\begin{equation*}
	J(x) = \sum\limits_{z \in \mathcal{N}} z \cdot m \left\langle z | x\right\rangle.
	\end{equation*}
	If $x = x_{1}+x_{2}$, where $x_{1} \in \mathcal{H}^{Jm}_{+}, x_{2} \in \mathcal{H}^{Jm}_{-}$, then
	\begin{equation*}
	T(x) = T_{+}(x_{1}) - T_{+}(x_{2} \cdot n) \cdot n.
	\end{equation*}
	By Theorem \ref{complexbound}, there exist a  measure space $(\Omega, \mu), $ a $ \mathbb{C}_{m}$- valued $\mu$- measurable function $\phi$ on $\Omega$ and a unitary operator $U_{+} \colon \mathcal{H}^{Jm}_{+} \to L^{2}( \Omega ; \mathbb{C}_{m}; \mu)$ such that
	\begin{equation*}
	T_{+} = U_{+}^{*}L_{\phi}U_{+},
	\end{equation*}
	where $L_{\phi} \colon L^{2}(\Omega; \mathbb{C}_{m}; \mu) \to L^{2}(\Omega; \mathbb{C}_{m}; \mu)$ is defined by
	\begin{equation*}
	L_{\phi}(g)(t) = \phi(t) \cdot g(t), \; \text{for all}\; g \in L^{2}(\Omega; \mathbb{C}_{m}; \mu).
	\end{equation*}
	By Theorem \ref{extension}, we have $\widetilde{L}_{\phi} \colon L^{2}(\Omega; \mathbb{H}; \mu) \to L^{2}(\Omega; \mathbb{H}; \mu)$ given by
	\begin{equation*}
	\widetilde{L}_{\phi}(g+h\cdot n) = L_{\phi}(g) + L_{\phi}(h) \cdot n, \; \text{for all}\; g, h \in L^{2}( \Omega ; \mathbb{C}_{m}; \mu).
	\end{equation*}
	Let  $\widetilde{L}_{\phi}:= M_{\phi}$. It is clear that ${M}_{\phi}$ is a right $\mathbb{H}$- linear and  ${M}_{\phi}|_{L^{2}( \Omega ; \mathbb{C}_{m}; \mu)} = L_{\phi}$.
	For $h = h_{1}+h_{2} \cdot n \in L^{2}( \Omega ; \mathbb{H}; \mu)$ and $ x \in \Omega $, we have
	\begin{align*}
	{M}_{\phi}(h_{1}+h_{2} \cdot n)(x) &= L_{\phi}(h_{1})(x) + L_{\phi}(h_{2})(x) \cdot n \\
	&= \phi(x)\cdot h_{1}(x)+ \phi(x) \cdot h_{2}(x) \cdot n \\
	&= \phi(x)(h_{1}(x) + h_{2}(x) \cdot n) \\
	&= \phi \cdot (h_{1} + h_{2})(x).
	\end{align*}
	That is ${M}_{\phi}$  is a multiplication operator induced by $\phi$. By Theorem \ref{extension1}, $U_{+}$ has a unique extension $U \colon \mathcal{H} \to L^{2}(\Omega; \mathbb{H}; \mu)$ such that
	\begin{equation*}
	U(x_{1}+x_{2})= U_{+}(x_{1}) - U_{+}(x_{2} \cdot n) \cdot n, \; \text{for all}\; x_{1} \in \mathcal{H}^{Jm}_{+}, x_{2} \in \mathcal{H}^{Jm}_{-}.
	\end{equation*}
	Let $x,y\in \mathcal{H}$. Then $x = x_{1}+x_{2}, y = y_{1} + y_{2}$, where $x_{1},y_{1} \in \mathcal{H}^{Jm}_{+}, x_{2},y_{2} \in \mathcal{H}^{Jm}_{-}$. Consider,
	\begin{align*}
	\Big\langle U^{*}{M}_{\phi}U(x) \; \Big|\; y\Big\rangle	 &= \left\langle M_{\phi}Ux \; \Big|\; Uy \right\rangle \\
	&= \left\langle M_{\phi}(U_{+}(x_{1})- U_{+}(x_{2}\cdot n)\cdot n) \; \Big|\; U_{+}y_{1}-U_{+}(y_{2}\cdot n)\cdot n\right\rangle \\
	&= \left\langle U_{+}^{*}L_{\phi}U_{+}(x_{1}) \; \Big|\; y\right\rangle  - \left\langle U_{+}^{*}L_{\phi}U_{+}(x_{2}\cdot n)\cdot n \; \Big|\;y_{2} \right\rangle \\
	&= \left\langle Tx_{1}-T(x_{2}\cdot n)\cdot n \; \Big|\; y\right\rangle \\
	&= \left\langle Tx \; \Big|\; y\right\rangle.
	\end{align*}
	
	Hence $U^{*}M_{\phi}U = T$.
	
	Proof of (\ref{norm}):
	It is clear from Proposition \ref{extension}(1), that $\|T\| = \|T_{+}\|$ and since $\|T_{+}\| = ess\ sup(|\phi|)$, we have $ \|T\| = ess\ sup (|\phi|)$.
	
	Proof of (\ref{spectrum}): We know that $\sigma(T_{+}) = ess\ ran(\phi)$. By Lemma \ref{extnspectrum}, we have
	\begin{equation*}
	\sigma_{S}(T) = \bigcup\limits_{\lambda \; \in\; ess\ ran(\phi)} [\lambda].
	\end{equation*}
	If $\mathcal{H}$ is separable, by Lemma \ref{separable}, $\mathcal{H}^{Jm}_{+}$ is separable.  Therefore by Theorem \ref{complexbound}, $(\Omega, \mu)$ is $\sigma$- finite.
\end{proof}
%%%%%%%%%%%%%%%%%%%%%%%%%%%%%%%%%%%%%%%%%%%%%%%%%%%%%%%%%%%%%%%%%%%%%%%%%%%%%%%%%%%%%%%%%%%%%%%%%%%%%%%%%%%%%%%%%%%%%%%%%%%%%%%%%%%%%%%%%%%%%%%%%%%%%%5

\begin{corollary}\label{quaternion}
	Let $m \in \mathbb{S},\; T \in \mathcal{B}(\mathcal{H})$ be a normal and  $\phi$ be as in  Theorem \ref{quaternionbounded}. Then the following hold true:
	\begin{enumerate}
		\item $T$ is anti self-adjoint if and only if $\phi$ is purely imaginary.
		\item $T$ is unitary if and only if $|\phi| = 1$ in $\mu$- a.e.
	\end{enumerate}
\end{corollary}
\begin{proof} By Theorem \ref{quaternionbounded}, we have $T = U^{*}M_{\phi}U$.
	
	\noindent Proof of $(1):$
	\begin{align*}
	T \; \text{is anti self-adjoint} &\Leftrightarrow T^{*}= -T \\
	&\Leftrightarrow U^{*}M_{\overline{\phi}}U = U^{*}M_{-{\phi}}U \\
	&\Leftrightarrow \overline{\phi} = - \phi.
	\end{align*}
	Proof of $(2):$	
	{\begin{align*}
		T \; \text{is unitary} \; &\Leftrightarrow T^{*}T = TT^{*} = I \\
		&\Leftrightarrow U^{*}M_{|\phi|^{2}}U = I\\
		& \Leftrightarrow U^{*}M_{|\phi|^{2}}U = U^{*}M_{1}U \; ;\; \Big(\text{Here $M_{1}$ is the identity on } L^{2}(\Omega; \mathbb{H}; \mu)\Big)\\
		&\Leftrightarrow |\phi| = 1, \; \text{ $\mu$ - a.e}.  \qedhere
		\end{align*} }
\end{proof}
%%%%%%%%%%%%%%%%%%%%%%%%%%%%%%%%%%%%%%%%%%%%%%%%%%%%%%%%%%%%%%%%%%%%%%%%%

The following result is well known (see \cite{Ghiloni,viswanath} for details). Here we prove it using a different method.
\begin{corollary}
	Let $T \in \mathcal{B}(\mathcal{H})$ be normal. Then $T$ and $T^{*}$ are unitarily equivalent.
\end{corollary}
\begin{proof}
	For a fixed $m \in \mathbb{S}$, by Theorem \ref{quaternionbounded}, there exists a Hilbert basis $\mathcal{N}_{m}$, a measure space $(\Omega, \mu) ,\ \mu$- measurable $\mathbb{C}_{m}$- valued function $\xi$ on $\Omega$ and a unitary $U \colon \mathcal{H} \to L^{2}(\Omega;\mathbb{H}; \mu)$ so that if $T$ is expressed with respect to $\mathcal{N}_{m}$, we have $T = U^{*}M_{\xi}U$. This implies $T^{*} = UM_{\overline{\xi}}U^{*}$.
	
	For $x \in \Omega$, define $\phi(x)=\frac{\xi(x)\cdot n}{|\xi(x)|}$, for all $x\in \Omega$. Here $n\in \mathbb S$ is such that $m\cdot n=-n\cdot m$. Clearly, $\phi$ is non-zero almost everywhere w.r.to $\mu$,  ess  sup $(|\phi|) = 1$ and $M_{\phi}$ is a right linear, unitary operator. Moreover,
	$M_{\phi}^{*}M_{\overline{\xi}}M_{\phi} = M_{{\xi}}$ and
	\begin{equation*}
	T = U^{*}M_{\xi} U = U^{*}M_{\phi}^{*} M_{\overline{\xi}}M_{\phi}U= U^{*}M_{\phi}^{*}U T^{*} U^{*}M_{\phi}U.
	\end{equation*}
	Let $V = U^{*}M_{\phi}U$. Then $V$ is unitary and $T = V^{*}T^{*}V$. Hence $T$ and $T^{*}$ are unitarily equivalent.
\end{proof}
%%%%%%%%%%%%%%%%%%%%%%%%%%%%%%%%%%%%%%%%%%%%%%%%%%%%%%%%%%%%%%%%%%%%%%%%
We illustrate our main theorem by the following example.
\begin{eg}\label{multiplicationex}
	Let $\phi (t) = (i-j-k)t$, for all $t \in [0,1]$ . Then $\phi$ is essentially bounded measurable function  with the Lebesgue measure $\mu$ on $[0,1]$. Define  $M_{\phi} \colon L^{2}\left([0,1];\mathbb{H}; \mu\right) \to L^{2}([0,1];\mathbb{H}; \mu)$ by
	\begin{equation*}
	M_{\phi}(g)(t) = \phi(t) \cdot g(t)\; \text{for all}\; g \in L^{2}([0,1];\mathbb{H}; \mu).
	\end{equation*}
	Fix $i \in \mathbb{S}$. We show that $M_{\phi}$ is unitarily equivalent to a multiplication operator on $L^{2}([0,1];\mathbb{H}; \mu)$ induced by some complex valued measurable function.
	
	Define $U \colon L^{2}([0,1];\mathbb{H}; \mu) \to L^{2}([0,1];\mathbb{H}; \mu)$ by
	\begin{equation*}
	U(g)(t) = \frac{(\sqrt{3}+1)- j + k}{\sqrt{6+2 \sqrt{3}}} \cdot g(t), \; \text{for all}\; g \in L^{2}([0,1];\mathbb{H}; \mu).
	\end{equation*}
	It follows that
	\begin{equation*}
	U^{*}(h)(t) = \frac{(\sqrt{3}+1)+ j - k}{\sqrt{6+2 \sqrt{3}}},\;  \text{for all} \;h \in L^{2}([0,1];\mathbb{H}; \mu).
	\end{equation*}
	It can be easily verified that $U$ is unitary.
	
	Define $\eta(t) = \sqrt{3}\, i t $, for all  $t\in [0,1]$. Clearly, $\eta$ is a complex valued essentially bounded measurable function. Also $\eta$ induces a bounded multiplication operator $M_{\eta}$ on $L^{2}([0,1];\mathbb{H}; \mu)$. We prove that $M_{\phi}$ is unitarily equivalent to $M_{\eta}$. For all $g \in L^{2}([0,1];\mathbb{H}; \mu)$, we have
	\begin{align*}
	U^{*}M_{\eta}U(g)(t) &= U^{*}M_{\eta}\frac{(\sqrt{3}+1)- j + k}{\sqrt{6+2 \sqrt{3}}} \cdot g(t)\\
	&= U^{*}\sqrt{3}it\cdot \frac{(\sqrt{3}+1)- j + k}{\sqrt{6+2 \sqrt{3}}} \cdot g(t)\\
	&=\frac{(\sqrt{3}+1)+ j - k}{\sqrt{6+2 \sqrt{3}}} \cdot \sqrt{3}it\cdot \frac{(\sqrt{3}+1)- j + k}{\sqrt{6+2 \sqrt{3}}} \cdot g(t)\\
	&= (i-j-k)t \cdot g(t) \\
	&= M_{\phi}(g)(t).
	\end{align*}
\end{eg}
\begin{note}
	In Example \ref{multiplicationex}, we have shown that the multiplication operator $M_{\phi}$ (induced by a $\mathbb{H}$- valued function $\phi$)  is unitarily equivalent to multiplication operator induced by $\mathbb{C}$- valued function.
\end{note}
%%%%%%%%%%%%%%%%%%%%%%%%%%%%%%%%%%%%%%%%%%%%%%%%%%%%%%%%%%%%%%%%%%%%%%%%%%%%%%%%%%%%%%%%%%%%%%%%%%%%%%%%%%%%%%%%%%%%%%%%%%%%%%%%%%%%%%%%%%%%%%%%%%%%%%%%%%
\section{Spectral theorem: Unbounded operators}
%%%%%%%%%%%%%%%%%%%%%%%%%%%%%%%%%%%%%%%%%%%%%%%%%%%%%%%%%%%%%%%%%%%%%%%%%%%%%%%%%%%%%%%%%%%%%%%%%%%%%%%%%%%%%%%%%%%%%%%%%%%%%%%%%%%%%%%%%%%%%%%%%%%%%%%%%%
In this section we prove the spectral theorem (multiplication form) for unbounded quaternionic normal operators. First we restrict the given quaternionic normal operator to the slice Hilbert space, later establish the result via bounded transform and  Theorem \ref{complexbound}.

\begin{theorem}\label{unboundedmultiplication}
	Let $T \in \mathcal{C}(\mathcal{H})$ be normal and $ m \in \mathbb{S}$.  Then there exists the following:
	
	\begin{enumerate}
		\item[(a)] a Hilbert basis $\mathcal{N}_{m}$ of $\mathcal{H}$
		\item[(b)] a  measure space $(\Omega_{0}, \nu)$
		\item[(c)] a unitary operator $V \colon \mathcal{H} \to L^{2}(\Omega_{0}; \mathbb{H}; \nu)$ and
		\item[(d)]  a $\nu$- measurable function $ \phi \colon \Omega_{0} \to \mathbb{C}_{m}$
	\end{enumerate}
	so that if $T$ is expressed with respect to $\mathcal{N}_{m}$, then
	\begin{equation*}
	Tx = U^{*}M_{\eta} Ux, \; \text{for all } \; x\in \mathcal{D}(T),
	\end{equation*}
	where $M_{\eta}$ is quaternionic multiplication operator on  $L^{2}(\Omega_{0}; \mathbb{H}; \nu)$ induced by $\eta$, with the domain
	\begin{equation*}
	\mathcal{D}(M_{\eta}) = \Big\{g \in L^{2}(\Omega_{0}; \mathbb{H}; \nu)| \eta \cdot g \in L^{2}(\Omega_{0}; \mathbb{H}; \nu)\Big\}.
	\end{equation*}
\end{theorem}
\begin{proof}
	By Proposition \ref{extension}, there exists a unique $\mathbb{C}_{m}$- linear operator $T_{+} \colon \mathcal{D}(T) \cap \mathcal{H}^{Jm}_{+} \to \mathcal{H}^{Jm}_{+}$ such that $\widetilde{T}_{+} = T$. Let $\mathcal{N}_{m}$ be a Hilbert basis  of $\mathcal{H}^{Jm}_{+}$. Then  by Remark \ref{J},   $\mathcal{N}_{m}$ is Hilbert basis for $\mathcal{H}$ and Moreover,
	\begin{equation*}
	J(x) = \sum\limits_{z \in \mathcal{N}} z \cdot m \left\langle z | x\right\rangle.
	\end{equation*}
	It is clear that $\widetilde{\mathcal{Z}}_{T_{+}} = \mathcal{Z}_{T}$ and $\mathcal{Z}_{T_{+}}$ is bounded $\mathbb{C}_{m}$- linear operator. By Theorem \ref{complexbound}, there is a measure space $(\Omega; \mu)$, a unitary operator $U_{+} \colon \mathcal{H}^{Jm}_{+} \to L^{2}(\Omega; \mathbb{C}_{m}; \mu)$ and a $\mu$- measurable function $\phi$ such that
	\begin{equation}\label{ztransform}
	\mathcal{Z}_{T_{+}} = U_{+}^{*}L_{\phi}U_{+}.
	\end{equation}
	Here $\Omega = \sigma(\mathcal{Z}_{T_{+}})$ and $\phi(z) = z$, for all $z \in \Omega$. Define  $ \xi \colon \Omega \to \mathbb{C}_{m}$ by
	\begin{equation*}
	\xi(p) = p (1-|p|^{2})^{-\frac{1}{2}},\; \text{for all}\; p \in \Omega.
	\end{equation*}
	Then $\xi$ is  Borel measurable function such that
	\begin{equation*}
	\mu \big(\{x \in \Omega: \xi(x) = \infty\}\big) = 0
	\end{equation*}
	By the Borel functional calculus for bounded
	$\mathbb{C}_{m}$- linear operator $Z_{T_{+}}$, we get
	\begin{equation}\label{funcionalcalculus}
	\xi(\mathcal{Z}_{T_{+}}) = \mathcal{Z}_{T_{+}}(I - \mathcal{Z}_{T_{+}}^{*}\mathcal{Z}_{T_{+}})^{-\frac{1}{2}}= T_{+}.
	\end{equation}
	By Equations (\ref{ztransform}) and (\ref{funcionalcalculus}), we have
	\begin{align*}
	T_{+} &= U_{+}^{*}L_{\phi}U (I - U_{+}^{*}L_{|\phi|^{2}}U_{+})^{-\frac{1}{2}}\\
	& = U_{+}^{*}L_{\phi(1-|\phi|^{2})^{-\frac{1}{2}}}U_{+}.
	\end{align*}
	Let us denote $\psi = \phi(1-|\phi|^{2})^{-\frac{1}{2}}$. Then
	\begin{equation}\label{Tplus}
	T_{+}x = U_{+}^{*}L_{\psi}U_{+}x, \; \text{for all } \; x \in \mathcal{D}(T_{+}).
	\end{equation}
	This implies that $U_{+}(\mathcal{D}(T_{+})) \subset \mathcal{D}(L_{\psi})$.
	It is clear $\sigma(T_{+}) = ess \ ran (\psi) = \Omega_{0}$ (say). Define a measure on $\Omega_{0}$ as $\nu(S) = \mu(\xi^{-1}(S))$, for every Borel subset S in $\Omega_{0}$.
	
	If $\eta(z) = z$ on $\Omega_{0}$, then $L_{\eta} \colon \mathcal{D}(L_{\eta}) \to L^{2}(\Omega_{0}; {\mathbb{C}_{m}}; \nu)$ defines a $\mathbb{C}_{m}$- linear operator in $L^{2}(\Omega; \mathbb{C}_{m}; \nu)$ with the domain $\mathcal{D}(L_{\eta}) = \{g \in L^{2}(\Omega_{0}; \mathbb{C}_{m}; \nu)| \eta \cdot g \in L^{2}(\Omega_{0}; \mathbb{C}_{m}; \nu)\}$. We establish a unitary between $L^{2}(\Omega; \mathbb{C}_{m}; \mu) $ and $L^{2}(\Omega_{0}; \mathbb{C}_{m}; \nu) $ as follows:
	Define $\pi \colon L^{2}(\Omega; \mathbb{C}_{m}; \mu) \to L^{2}(\Omega_{0}; \mathbb{C}_{m}; \nu) $ by
	\begin{equation*}
	\pi(g) = g \circ \xi^{-1},\; \text{for all} \; g \in L^{2}(\Omega; \mathbb{C}_{m}; \mu).
	\end{equation*}
	We claim that $\pi$ is  unitary. For $g \in L^{2}(\Omega; \mathbb{C}_{m}; \mu)$, we have
	\begin{align*}
	\|\pi(g)\|^{2} = \int\limits_{\Omega_{0}}|(g\circ \xi^{-1})(s)|^{2} d\nu(s)
	&= \int\limits_{\Omega_{0}}|g(\xi^{-1}(s))|^{2} d\nu(s)\\
	&= \int\limits_{\Omega} |g(t)|^{2}d\mu(t)\\
	&= \|g\|^{2}.
	\end{align*}
	This shows that $\pi$ is one to one as well as well defined. If $h \in L^{2}(\Omega_{0}; \mathbb{C}_{m}; \nu)$,  then $h\circ \xi \in L^{2}(\Omega; \mathbb{C}_{m}; \mu)$ such that $\pi(h\circ \xi) = h$. This implies $\pi$ is onto.
	
	Let $g \in L^{2}(\Omega; \mathbb{C}_{m}; \mu)$ and $h \in L^{2}(\Omega_{0}; \mathbb{C}_{m}; \nu)$. Then
	\begin{align*}
	\left\langle \pi(g) | h\right\rangle = \int\limits_{\Omega_{0}} \overline{\pi(g)(t)} h(t) d\nu(t)
	&= \int\limits_{\Omega_{0}} \overline{g(\xi^{-1}t)} h(t) d\nu(t)\\
	&= \int\limits_{\Omega} g(s) h(\xi(s)) d\mu(s)\\
	&= \left\langle g | h \circ \xi \right\rangle.
	\end{align*}
	This implies $\pi^{*}(h) = h\circ \xi$, for all $h \in L^{2}(\Omega_{0}; \mathbb{C}_{m}; \nu)$. It can be verified that $\pi^{*}\pi = \pi \pi^{*} = I$. First we express $L_{\eta}$ in terms of $L_{\psi}$. Later we construct unitary $V_{+}$ between $\mathcal{H}^{Jm}_{+}$ and $L^{2}(\Omega_{0}; \mathbb{C}_{m}; \nu)$.
	Consider
	\begin{align*}
	(\pi^{*}L_{\eta}\pi)(g)(x) = \pi^{*}L_{\eta}(g \circ \xi^{-1})(x)
	&= \pi(\eta \cdot g\circ \xi^{-1})(x)\\
	&= \eta \cdot (g\circ \xi^{-1}) \circ \xi (x)\\
	&= (\eta \circ \xi)(x) \cdot g(x)\\
	&= \psi(x) \cdot g(x)  \\
	&= L_{\psi} (g)(x).
	\end{align*}
	This shows $ \pi^{*}L_{\eta}\pi = L_{\psi} $.
	Define $V_{+} \colon \mathcal{H}^{Jm}_{+} \to L^{2}(\Omega_{0}; \mathbb{C}_{m}; \nu)$ by
	\begin{equation*}
	V_{+} = \pi \circ U_{+}.
	\end{equation*}
	
	The following  diagram helps in understanding the construction of unitary operators.
	\begin{center}
		\begin{tikzpicture}[scale= 3.5]
		\node (A) at (0,1) {$\mathcal{D}(L_{\psi}) \subseteq L^{2}(\Omega; \mathbb{C}_{m}; \mu) $};
		\node (B) at (1,1) {$L^{2}(\Omega; \mathbb{C}_{m}; \mu)$};
		\node (C) at (0,0) {$\mathcal{D}(L_{\eta})\subseteq L^{2}(\Omega_{0}; \mathbb{C}_{m}; \nu)$};
		\node (D) at (1,0) {$L^{2}(\Omega_{0}; \mathbb{C}_{m}; \nu)$};
		\node (E) at (-1,0) {$\mathcal{H}^{Jm}_{+}$};
        \path[->] (A) edge node[above]{$L_{\psi}$} (B)
       edge node[right]{$\pi$} (C)
     	(D) edge node[right]{$\pi^{*}$} (B)
		(C) edge node[above]{$L_{\eta}$} (D)
		(E) edge node[left]{$U_{+}$} (A)
		(E) edge node[above]{$V_{+}$} (C);
		\end{tikzpicture}
	\end{center}

	We claim that $V_{+}$ is unitary and $V_{+}^{*}L_{\eta}V_{+}x = T_{+}x$ for all $x\in \mathcal{D}(T_+)$. Since $\pi$ and $U_{+}$ are unitary, it can be easily seen that $V_{+}^{*}V_{+} = V_{+}V_{+}^{*}=I$. Then by Equation (\ref{Tplus}), we have
	\begin{align*}
	T_{+}x &= U_{+}^{*}\pi^{*}L_{\eta}\pi U_{+}x\\
	&= V_{+}^{*}L_{\eta}V_{+}x,
	\end{align*}
	for all $x\in \mathcal D(T_+)$.
	
	Now extend the operator $T_{+}$ to the operator $T$ in $\mathcal{H}$ by using Proposition \ref{extension} and Theorem \ref{extension1}. The rest of the proof follows in the similar lines as in the case of bounded operators. Let $\widetilde{L_{\eta}} = M_{\eta}$ and $\widetilde{V_{+}} = V$. Then by extension of $T_{+}$ we get
	\begin{equation*}
	Tx = V^{*} {M}_{\eta} Vx,\; \text{for all}\; x\in \mathcal{D}(T),
	\end{equation*}
	where ${M}_{\eta}$ is the quaternionic multiplication operator in $L^{2}(\Omega_{0}; \mathbb{H}; \nu)$.
\end{proof}

\section*{Acknowledgment}
	\thanks{The second author is thankful to INSPIRE (DST) for the support in the form of fellowship (No. DST/ INSPIRE Fellowship/2012/IF120551), Govt of India.}

%% Important: Do not put any empty line here.
%\affiliationtwo{% in this example, one author has two addresses}
%   T. Hird\\
%   Previous postal address where
%     the research was performed and\\
%   Country
%   \email{hird@university.ac.uk}}
%% Important: Do not put any empty line here.
%% Use \affiliationthree{} for any address positioned under \affiliationone
%% Use \affiliationfour{}  for any address positioned under \affiliationtwo
%\affiliationthree{~} %inserts a space to make this field empty
%\affiliationfour{%
%   Current address:\\
%   Present long-term address\\
%   Country
%   \email{t.hird@institution.edu}}
%%

%%%%%%%%%%%%%%%%%%%%%%%%%%%%%%%%%%%%%%%%%%%%%%%%%%%%%%%%%%%%%%%%%%%%%%%%%%%%%%%%%%%%%%%%%%%%%%%%%%%%%%%%%%%%%%%%%%%%%%%%%%%
\end{document}